\documentclass[article,hidelinks,onefignum,onetabnum]{siamart251216}
\usepackage{tikz}
\usetikzlibrary{decorations.pathmorphing}
\usepackage{bbm}
\usepackage[utf8]{inputenc}  
\usepackage{amsrefs}
\usepackage{amssymb,amsmath}
\usepackage{enumerate}
\usepackage[shortlabels]{enumitem}
\usepackage{mathtools}
\usepackage{graphicx}
\DeclareRobustCommand\dotted{\tikz[baseline=-0.6ex]\draw[thick,dotted] (0,0)--(0.54,0);}
\DeclareRobustCommand\dashed{\tikz[baseline=-0.6ex]\draw[thick,dashed] (0,0)--(0.54,0);}

\DeclareRobustCommand\oline{%
\tikz[baseline=-0.6ex]{
  \draw[thick] (0,0)--(0.5,0);
  \filldraw (0.25,0) circle (1.5pt);
}}
\DeclareRobustCommand\chainn{\tikz[baseline=-0.6ex]\draw[thick,dash dot] (0,0)--(0.5,0);}
\usepackage{url}
\usepackage{color}  
\definecolor{darkred}{rgb}{0.6,0,0}
\definecolor{darkgreen}{rgb}{0,0.5,0}
\definecolor{darkmagenta}{rgb}{0.5,0,0.5}
  
\usepackage[capitalize,nameinlink]{cleveref}

\makeatletter
\def\@cite#1#2{\textup{[{#1\if@tempswa , #2\fi}]}}
\DeclareMathOperator{\loc}{loc}
\makeatother

\DeclarePairedDelimiter\abs{\lvert}{\rvert}

   \DeclarePairedDelimiterX\Set[1]\{\}{%
      
      #1
   }

\DeclareMathOperator{\bv}{BV}
\newcommand{\R}{\mathbb R}

\newcommand{\dott}{\, \cdot\,}

 \usepackage[sort]{cite}
\usepackage{bigints}
\newcommand{\Z}{\mathbb{Z}}
\DeclareMathOperator{\lip}{Lip}
\newcommand{\sgn}{\mathop\mathrm{sgn}}
\newcommand{\norma}[1]{{\left\|#1\right\|}}

\renewcommand{\d}[1]{\mathinner{\mathrm{d}{#1}}}
\DeclareMathOperator{\TV}{TV}
\DeclareMathOperator{\lipR}{Lip(\R)}
\newcommand{\D}{\Delta}

\numberwithin{equation}{section}     
\allowdisplaybreaks

\usepackage{lipsum}
\usepackage{amsfonts}
\usepackage{graphicx}
\usepackage{epstopdf}
\usepackage{algorithmic}
\ifpdf
  \DeclareGraphicsExtensions{.eps,.pdf,.png,.jpg}
\else
  \DeclareGraphicsExtensions{.eps}
\fi

\newsiamremark{remark}{Remark}
\newsiamremark{hypothesis}{Hypothesis}
\crefname{hypothesis}{Hypothesis}{Hypotheses}
\newsiamthm{claim}{Claim}
\newsiamremark{fact}{Fact}
\crefname{fact}{Fact}{Facts}

\headers{Stability estimates: memory}{}

\title{Stability estimates for systems of nonlocal balance laws with memory
\thanks{
\funding{This work was  supported by the  University Grants Commission, India NTA under Grant 221610073569 for NKA, AA's Seed Money Grant SM/08/2025-26 and the ARG Matrics Grant ANRF/ARGM/2025/001976/MTR from the Anusandhan National Research Foundation (ANRF), India, as well as GV's INSPIRE Faculty Fellowship (IFA24-MA215) from the Department of Science and Technology (DST), Government of India. The authors
also acknowledge the hospitality of the Department of Mathematics, Indian Institute of Space Science and Technology,
Thiruvananthapuram, where part of this work was completed during AA’s and GV's visit.}}}

\author{
Aekta Aggarwal\thanks{Operations Management and Quantitative Techniques,
Indian Institute of Management Indore\\
Prabandh Shikhar, Rau--Pithampur Road\\
Indore 453556, Madhya Pradesh, India
(\email{aektaaggarwal@iimidr.ac.in}).}
\and
N.~K.~Aswini\thanks{Department of Mathematics\\
Indian Institute of Space Science and Technology\\
Thiruvananthapuram 695547, Kerala, India
(\email{aswinink.23@res.iist.ac.in}, \email{sarvesh@iist.ac.in}).}
\and
Sarvesh Kumar\footnotemark[3]
\and
Ganesh Vaidya\thanks{Department of Mathematics\\
Indian Institute of Science\\
Bangalore 560012, Karnataka, India
(\email{vaidyaganesh@iisc.ac.in}).}
}

\usepackage{amsopn}

\begin{document}\maketitle\begin{abstract}
    In this work, we investigate entropy solutions for a class of systems of nonlocal {balance laws in which the convective flux and the source involves terms where the state variable convolved with kernels} in both spatial and temporal variables. This formulation captures the dependence of the flux on the solution within its spatial neighborhood (spatial nonlocality) as well as on its past states (temporal nonlocality), thereby incorporating memory effects. The resulting systems are coupled through these nonlocal interactions. We establish stability estimates for entropy solutions with respect to perturbations in the flux, the spatial and temporal kernels, and the initial data for the corresponding initial value problems. Finally, we present numerical experiments to illustrate the theoretical results and to highlight the influence of memory and source terms on the solution dynamics.
\end{abstract}


\section{Introduction}Conservation laws with nonlocal interactions provide a natural framework for modeling systems in which the evolution at a given point depends not only on the instantaneous state but also on spatial averages. PDEs of the form are known to model such dynamics and read as
\begin{equation}\label{nls1}
\partial_t u +\partial_x \Big(f(u,\int_{\mathbb{R}} u(t,\xi)\mu(x-\xi)\, d\xi ) \Big)=0, \quad (t,x) \in Q_T:=(0,T)\times \R.
\end{equation} Here $\mu$ is a spatial horizon kernel and models applications where the flux at a given point may depend 
not only on the local state but also on averaged quantities over a finite 
interaction horizon. These models arise naturally in applications such as crowds~\cite{CGL2012,ACG2015,BG2016,BHL2023,FGKP2022,AHV2023_1,AHV2024}, sedimentation~\cite{BBKT2011}, laser technology~\cite{CM2015}, granular media~\cite{AS2012} and conveyor belt dynamics~\cite{GHS+2014}.  Their well-posedness and quantitative stability results for entropy solutions with respect to initial data and parameters of the model like $f,\mu$ and radius of $\mu$, have been extensively studied in the last decade~\cite{CK24,FGR2021,BFK2022,AV2023,FCV2023,keimer2021discontinuous,CG2019,KLS2018,ANT2007,AS2012,BHL2023,ACT2015,AG2016,BG2016,FGKP2022,CGL2012,AHV2023_1,AHV2024,AHV2023,ACG2015,BBKT2011,GHS+2014,CG2023,CM2015,
colombo2012,
LecureuxMercier2011,
chiarello2019,
karlsen2000,
CR2018,
colombo2025general,
colombo2016biological,
colombo2009stability,
goatin2019well,
rossi2020well,
keimer2021discontinuous,
keimer2025optimal,
friedrich2023numerical,
pflug2023discontinuous
}. 
On the contrary, many real-world systems exhibit dynamics in which the present state depends on the past history, giving rise to \emph{memory effects}. Such hereditary behavior appears in applications ranging from viscoelasticity~\cite{Dafermos1970}, gas transport in porous media~\cite{Shi2003}, and subsurface transport processes~\cite{Haggerty1995}. A prototypical conservation law with memory reads
\begin{equation}\label{eq:memory_general}
\partial_t u
+ \partial_x \Big( f\big(u,\int_0^t u(\tau,x)\Gamma(t-\tau)\, d\tau \big) \Big)
=0,
\end{equation}
where $\Gamma$ is a temporal kernel encoding past influence. While several works~\cite{Dafermos1970, D1987, c2008, CC2007, DAF2012, CHR2007, N2023, DHSS2023, P2014} establish well-posedness under restrictive structural assumptions, a general $L^1$ stability theory for conservation laws with temporal nonlocality remains largely open.

These challenges motivated the authors in \cite{AV2026} to combine nonlocal-in-space and nonlocal-in-time dynamics and introduce an explicit memory-dependent formulation capturing the history dependence of the solution through space--time convolution kernels. The present work goes beyond this setting by incorporating {nonlocal space--time source terms}, leading to the following class of systems of nonlocal balance laws with memory:
\begin{align}
    \label{nlm}
  \partial_t U^{k} +\partial_x \Big(f^k(U^k) \nu^k((\boldsymbol{\Theta} \circledast  \boldsymbol{U})^k)\Big) &=R^k (\boldsymbol{U},(\boldsymbol{\Upsilon} \circledast {\boldsymbol{U}})^k), \quad (t,x) \in Q_T,\\
  \label{init}
  U^{k}(0,x)&=U_0^{k}(x), \quad x \in \R,
 \end{align} where $k\in \mathcal{N}:=\{1, \ldots, N\}$, $T$ is the final time and the unknown is $\boldsymbol{U}=(U^{k})_{k\in\mathcal{N}}:[0,\infty)\times\mathbb{R}\to\mathbb{R}^N.$ 
Further, for every $j,k\in \mathcal{N}$, ${U}^k_0\in (L^1 \cap \bv) (\R;[0,1]),\ \Theta^{j,k}(t,x):=\mu^{j,k}(x) \Gamma^{j,k}(t),\ \Upsilon^{j,k}(t,x):=\eta^{j,k}(x) \theta^{j,k}(t)$, with $\boldsymbol{\mu}:=(\mu^{j,k})_{j,k\in\mathcal{N}},\ \boldsymbol{\Gamma}:=(\Gamma^{j,k})_{j,k\in\mathcal{N}},$ $\ \boldsymbol{\eta}:=(\eta^{j,k})_{j,k\in\mathcal{N}}$ and $\boldsymbol{\theta}:=(\theta^{j,k})_{j,k\in\mathcal{N}}$ being smooth $N\times N$ matrices. Further, for $\boldsymbol{Z}=\boldsymbol{\Theta}$ and $\boldsymbol{\Upsilon}$, and for every $k\in \mathcal{N},$ $(\boldsymbol{Z}\circledast \boldsymbol{U})^k:= (Z^{j,k}*U^j)_{j\in\mathcal{N}}$,  where for every $(t,x)\in \overline{Q}_T$ 
\begin{align}\label{mc}
\begin{split}
   (\Theta^{j,k}*U^j)(t,x)&:=\displaystyle\int_0^t\int_{\R} U^j(\tau,\xi)\mu^{j,k}(x-\xi)\Gamma^{j,k}(t-\tau)\d \xi \d \tau ,\ j\in\mathcal{N},\\
    (\Upsilon^{j,k}*U^j)(t,x)&:=\displaystyle\int_0^t\int_{\R} U^j(\tau,\xi)\eta^{j,k}(x-\xi)\theta^{j,k}(t-\tau)\d \xi \d \tau ,\ j\in\mathcal{N}.
\end{split}
\end{align}
Additionally, we assume that
\begin{enumerate}[(\textbf{H\arabic*})]
\item \label{H1A}$f^k \in  \lip(\R)$  with $ f^k(0)=0=f^k(1)$.
 \item \label{H2A}$\nu^k \in (C^2 \cap \bv \cap \, W^{2,\infty}) (\R^N,\R)$.
 \item \label{H3A}
The space kernels ${\mu}^{j,k},{\eta}^{j,k}\in C^1(\R) \cap W^{1,\infty}(\R)$, and the time kernels $\Gamma^{j,k},$ $\theta^{j,k}\in L^{\infty}([0,\infty);\R^+).$ 
  \item \label{H4A} $R^k \in  \lip(\R^{2N})$  with $ R^k(\boldsymbol{0},\boldsymbol{0})=0$.
 \end{enumerate} 
From a modeling perspective, the convective term captures the propagation of quantities whose velocity depends on both the local density and its past states, while the source term represents production, dissipation, or exchange effects that are themselves nonlocal in space and time. This framework is relevant, for example, to transport processes with delayed responses, multi-species interactions, and systems exhibiting relaxation or hysteresis effects, and can be seen as nonlocal-space-time extension of \cite{AHV2024,BFK2022,CK24,HR2019}. The well-posedness of similar conservation laws has also recently been studied in \cite{K2026}, where both spatial and temporal nonlocalities are incorporated at the level of the flux. However, the resulting PDE is treated primarily as an operator equation within a fixed-point framework, where the memory effect enter implicitly through a time-dependent flux operator and is controlled via contraction arguments and entropy estimates, without isolating an explicit convolution structure in the evolution equation.

In the present work, we establish quantitative stability results for entropy solutions for \eqref{nlm}, showing continuous dependence on initial data as well as on the flux functions, velocity fields, nonlocal interaction kernels, and memory terms. In particular, we derive a Lipschitz-type $L^1$-estimate that explicitly captures the influence of the memory kernel and its interaction with spatial nonlocal effects. To the best of our knowledge, this is the first stability result for systems of nonlocal balance laws with space–time convolution structure of the type \eqref{nlm}. The present work serves as a companion to \cite{AV2026}, which established existence and uniqueness of entropy solutions for the source-free {\color{black}counterpart} of \eqref{nlm}. Together, these results yield the complete well-posedness theory for this class of systems. 

The remainder of the paper is organized as follows. In \S\ref{def}, we introduce the precise assumptions and notations. In \S\ref{uni}, we establish stability estimates with respect to initial data and nonlinear fluxes. In \S\ref{num}, we propose a first-order numerical scheme for the approximation of the initial value problem. Finally, in \S\ref{num1},  numerical experiments supporting the theory of the article are presented.

\section{Definitions and notation}\label{def}
In this section, we introduce the notations used in the sequel:
\begin{enumerate}
    \item For $\boldsymbol{Z}:=(Z^k)_{k\in\mathcal{N}}\in \R^N,$ let $\norma{\boldsymbol{Z}}:=\displaystyle\sum\limits_{k\in\mathcal{N}}\abs{Z^k}$ denote the usual $1$-norm.\\
    \item 
$
\norma{\boldsymbol{\Theta}}_{(L^{\infty}(\overline{Q}_T))^{N^2}}:= \max\limits_{j,k\in\mathcal{N}} \norma{\Theta^{j,k}}_{L^{\infty}(\overline{Q}_T)}.
$
\item If $\boldsymbol{\mu}\in C^1(\R;\R^{N^2})$, then $\boldsymbol{\mu}'=({\dot{\mu}}^{j,k})_{j,k\in\mathcal{N}}\in C(\R;\R^{N^2})$  denote the component-wise derivative.
   \item For $\boldsymbol{U}:\overline{Q}_T \rightarrow \R^N$ 
  and $\tau>0$, 
\begin{align*}
&|\boldsymbol{U}|_{(L^\infty_t\bv_x)^N}:=\max_{k\in\mathcal{N}}\sup_{t\in[0,T]} TV(U^k(t,\dott)),\\
&\quad 
\norma{\boldsymbol{U}}_{(L^{\infty}(\overline{Q}_T))^N}:=\max_{k\in\mathcal{N}} \norma{U^k}_{L^{\infty}(\overline{Q}_T)},\\ 
    &\quad 
\norma{\boldsymbol{U}}_{(L^1(\overline{Q}_T))^N}:=\sum\limits_{k\in\mathcal{N}}\norma{U^k}_{L^1(\overline{Q}_T)}.
\end{align*} 
\end{enumerate}
\vspace{-4mm}
Since $f^k$ is nonlinear, there can be multiple weak solutions of \eqref{nlm}-\eqref{init}, like in a local hyperbolic conservation law. Hence, an entropy condition is required to single out the unique solution.
\begin{definition}\label{def:entropy}
    A function $\textbf{U} \in (C([0,T];L^1(\R;[0,1]))\cap L^{\infty}([0,T];\bv(\R)))^N$  is an entropy solution of \eqref{nlm}-\eqref{init} with initial data $\textbf{U}_0$
if for each $(k,\alpha)\in \mathcal{N}\times\R$, and for all non-negative $\phi\in C_c^{\infty}([0,T)\times \R),$
\begin{multline} \label{kruz2}
\int_{Q_T}\left|U^k(t,x)- \alpha\right|\phi_t(t,x)  \d x \d t+\int_{\R} \left|U_0^k(x)- \alpha\right|\phi(0,x) \d x  \\ 
+ \int_{Q_T}\sgn ({U}^k(t,x)-\alpha) \nu^k((\boldsymbol{\Theta}\circledast \boldsymbol{U})^k(t,x))
(f^k({U}^k(t,x))-f^k(\alpha))\phi_x(t,x) \d{x} \d{t}\\ 
 -\int_{Q_T} f^k(\alpha) \sgn ({U}^k(t,x)-\alpha) \partial_x\nu^k((\boldsymbol{\Theta}\circledast \boldsymbol{U})^k(t,x))\phi(t,x)\d{x} \d{t}\\
 +\int_{Q_T} \sgn({U}^k(t,x)-\alpha)R^k (\boldsymbol{U},(\boldsymbol{\Upsilon} \circledast {\boldsymbol{U}})^k)(t,x)\phi(t,x)\d{x}\d{t} \geq 0. 
\end{multline}
\end{definition}
\section{Stability and Uniqueness}\label{uni} 
We now prove the stability of the  IVP \eqref{nlm}-\eqref{init} with respect to the flux, initial data and kernels of the convolution. More precisely, we have the following result:
\begin{theorem}[Stability Estimate]\label{theorem:stability}
Let $\boldsymbol{f}, \overline{\boldsymbol{f}} , \boldsymbol{\nu}, \overline{\boldsymbol{\nu}}, \boldsymbol{\mu}, \overline{\boldsymbol{\mu}},\boldsymbol{\Gamma}, \overline{\boldsymbol{\Gamma}},\boldsymbol{R}, \overline{\boldsymbol{R}}, \boldsymbol{\Upsilon}, $ and $\boldsymbol{\overline{\Upsilon}} $  satisfy \ref{H1A}--\ref{H4A}. Now fix initial data $\boldsymbol{U}_0, \boldsymbol{V}_0 \in (L^1(\R))^N$. Also let $\boldsymbol{U}$ be the entropy solution of the  IVP~\eqref{nlm}-\eqref{init} and $\boldsymbol{V}$ be the entropy solution of
\begin{align}
    \label{eq:coupled1}
  \partial_t V^{k} +\partial_x \Big(\overline{f}^k(V^k) \overline{\nu}^k((\overline{\boldsymbol{\Theta}} \circledast  \boldsymbol{V})^k)\Big) &=\overline{R}^k (\boldsymbol{V},(\overline{\boldsymbol{\Upsilon}} \circledast {\boldsymbol{V}})^k), \quad (t,x) \in Q_T,\\
  \label{eq:coupledic1}
  V^{k}(0,x)&=V_0^{k}(x), \quad x \in \R,
 \end{align}
with $k \in \mathcal{N}.$
Then, for any $t \in [0,T]$, the following estimate holds:

\begin{align}
\begin{split}
&\norma{\boldsymbol{U}(t,.)-\boldsymbol{V}(t,.)}_{(L^1(\mathbb{R}))^N}\le\Big(\norma{\boldsymbol{U}_0-\boldsymbol{V}_0}_{(L^1(\mathbb{R}))^N}+{\mathcal{C}_{1}}\abs{\overline{\boldsymbol f}-\boldsymbol{f}}_{(\lipR)^N}\\
    &+{\mathcal{C}_{2}}\norma{\overline{\boldsymbol\nu}-\boldsymbol{\nu}}_{(L^{\infty}(\R^N))^N}+{\mathcal{C}_{3}}\norma{\nabla \overline{\boldsymbol{\nu}}-\nabla \boldsymbol{\nu}}_{(L^{\infty}(\mathbb{R}^N))^{N^2}}+{\mathcal{C}_{4}}\norma{\overline{\boldsymbol{\Gamma}}-\boldsymbol{\Gamma}}_{(L^{\infty}(\R^+))^{N^2}}\\ \nonumber
&+{\mathcal{C}_{5}}\norma{\overline{\boldsymbol{\mu}}-\boldsymbol{\mu}}_{(W^{1,\infty}(\R))^{N^2}}+\mathcal{C}_{6}\norma{\overline{\boldsymbol{\eta}}-\boldsymbol{\eta}}_{(L^1(R))^{N^2}}\\ \nonumber
&\quad+\mathcal{C}_{7}\norma{\overline{\boldsymbol{\theta}}-\boldsymbol{\theta}}_{(L^1(\R^+))^{N^2}}+\mathcal{C}_{8}\abs{{\boldsymbol R}-\overline{\boldsymbol R}}_{\operatorname{Lip}({\R^{2N}})}\Big)\exp(\mathcal{C}_{9}t),
\end{split}
\end{align}
where ${\mathcal{C}_{1}}$ -- $\mathcal{C}_{9}$ are constants that depend on
$\boldsymbol{U}, \boldsymbol{V}, \boldsymbol{f},\overline{\boldsymbol{f}},
 \boldsymbol{\mu},\overline{\boldsymbol{\mu}},
 \boldsymbol{\Gamma},\boldsymbol{\overline{\Gamma}},
 \boldsymbol{\nu},
\overline{\boldsymbol{\nu}},{\boldsymbol R},$ and $\overline{\boldsymbol R}$.
\end{theorem}

\begin{proof}

  Let $(t,x) \in Q_T$ and $k\in\mathcal{N}.$
For $\phi = \phi(t,x,s,y) \in C_c^{\infty}(Q_T^2)$ and for a.e.~$(s,y) \in Q_T$, the entropy condition \eqref{kruz2} for ${U}^k(t,x)$ with $\alpha = {V}^k(s,y)$ can be rewritten as: \vspace{-2.5mm}
\begin{align}
\begin{split} \label{511}
0\leq&\int_{Q_T}\left|U^k(t,x)- {V}^k(s,y)\right|\phi_t  \d x \d t  \\ 
&+ \int_{Q_T}\sgn \Big({U}^k(t,x)-{V}^k(s,y)\Big) \nu^k((\boldsymbol{\Theta} \circledast  \boldsymbol{U})^k(t,x))
\Big(f^k({U}^k(t,x))\\
&\quad\quad\quad\quad\quad\quad\quad\quad\quad\quad\quad\quad\quad\quad\quad\quad\quad\quad\quad\quad\quad\quad-f^k({V}^k(s,y))\Big)\phi_x \d{x} \d{t}\\ 
 &-\int_{Q_T}\sgn \Big({U}^k(t,x)-{V}^k(s,y)\Big) f^k({V}^k(s,y))  \partial_x({\nu}^k((\boldsymbol{\Theta} \circledast  \boldsymbol{U})^k(t,x)))\phi\d{x} \d{t}\\
 &+\displaystyle\int_{Q_T} \sgn\Big({U}^k(t,x)-{V}^k(s,y)\Big)R^k (\boldsymbol{U},(\boldsymbol{\Upsilon} \circledast {\boldsymbol{U}})^k)(t,x)\phi\d{x}\d{t}. 
 \end{split}
\end{align}

Using the relation,
{\allowdisplaybreaks\begin{align*}
&\Big(f^k({U}^k(t,x))-f^k({V}^k(s,y))\Big)\nu^k((\boldsymbol{\Theta} \circledast  \boldsymbol{U})^k(t,x))\phi_x\\
&\quad\quad\quad\quad\quad\quad\quad-f^k({V}^k(s,y)) \partial_x({\nu}^k((\boldsymbol{\Theta} \circledast  \boldsymbol{U})^k(t,x)))\phi\\
&\quad=-\Big(\Big(f^k({V}^k(s,y))\nu^k((\boldsymbol{\Theta} \circledast  \boldsymbol{U})^k(t,x))-\overline{f}^k({V}^k(s,y))\overline{\nu}^k((\overline{\boldsymbol{\Theta}} \circledast  \boldsymbol{V})^k(s,y))\Big)\phi\Big)_x \\
&\quad\quad+\Big(f^k({U}^k(t,x)){\nu}^k((\boldsymbol{\Theta} \circledast  \boldsymbol{U})^k(t,x))-\overline{f}^k({V}^k(s,y))\overline{\nu}^k((\overline{\boldsymbol{\Theta}} \circledast  \boldsymbol{V})^k(s,y))\Big)\phi_x,
\end{align*}}
\eqref{511} can now be rewritten as:
{\allowdisplaybreaks
\begin{align}\label{611}
& -\displaystyle \int_{{\color{black}Q_T}} \abs{{U}^k(t,x)-{V}^k(s,y)}\phi_t \d{x} \d{t}\\\nonumber
&-\displaystyle \int_{{\color{black}Q_T}}\sgn \Big({U}^k(t,x)-{V}^k(s,y)\Big)\Big(f^k({U}^k(t,x)){\nu}^k((\boldsymbol{\Theta} \circledast  \boldsymbol{U})^k(t,x))\\ \nonumber
&\quad\quad\quad\quad\quad\quad\quad\quad\quad\quad\quad\quad\quad\quad\quad\quad-\overline{f}^k({V}^k(s,y))\overline{\nu}^k((\overline{\boldsymbol{\Theta}} \circledast  \boldsymbol{V})^k(s,y))\Big)\phi_x \d{x} \d{t} \\ \nonumber
&-\displaystyle \int_{{\color{black}Q_T}}\sgn\Big( {U}^k(t,x)-{V}^k(s,y)\Big)\Big(\left(\overline{f}^k({V}^k(s,y))\overline{\nu}^k((\overline{\boldsymbol{\Theta}} \circledast  \boldsymbol{V})^k(s,y))\right.\\ \nonumber
&\quad\quad\quad\quad\quad\quad\quad\quad\quad\quad\quad\quad\quad\quad\quad\quad\left.-f^k({V}^k(s,y)){\nu}^k((\boldsymbol{\Theta} \circledast  \boldsymbol{U})^k(t,x))\right)\phi\Big)_x\d{x} \d{t}\\ \nonumber
&-\displaystyle\int_{Q_T} \sgn\Big({U}^k(t,x)-{V}^k(s,y)\Big)R^k (\boldsymbol{U},(\boldsymbol{\Upsilon} \circledast {\boldsymbol{U}})^k)(t,x)\phi\d{x}\d{t}\leq 0.
\end{align}}
Now, repeating the same as above for the entropy condition \eqref{kruz2} for ${V}^k(s,y)$ with $\alpha={U}^k(t,x),$ we get, 
{\allowdisplaybreaks\begin{align}\label{6111}
&-\displaystyle \int_{{\color{black}Q_T}} \abs{{U}^k(t,x)-{V}^k(s,y)}\phi_s \d{y} \d{s}\\\nonumber
&-\displaystyle \int_{{\color{black}Q_T}}\sgn \Big({V}^k(s,y)-{U}^k(t,x)\Big)\Big(\overline{f}^k({V}^k(s,y))\overline{\nu}^k((\overline{\boldsymbol{\Theta}} \circledast  \boldsymbol{V})^k(s,y))\\ \nonumber
&\quad\quad\quad\quad\quad\quad\quad\quad\quad\quad\quad\quad\quad\quad\quad\quad-f^k({U}^k(t,x)){\nu}^k((\boldsymbol{\Theta} \circledast  \boldsymbol{U})^k(t,x))\Big)\phi_y \d{y} \d{s} \\\nonumber
&-\displaystyle \int_{{\color{black}Q_T}}\sgn\Big( {V}^k(s,y)-{U}^k(t,x)\Big)\Big(\Big(f^k({U}^k(t,x)){\nu}^k((\boldsymbol{\Theta} \circledast  \boldsymbol{U})^k(t,x))-\\ \nonumber
&\quad\quad\quad\quad\quad\quad\quad\quad\quad\quad\quad\quad\quad\quad\quad\quad\overline{f}^k({U}^k(t,x))\overline{\nu}^k((\overline{\boldsymbol{\Theta}} \circledast  \boldsymbol{V})^k(s,y))\Big)\phi\Big)_y\d{y} \d{s}\\ \nonumber
&-\displaystyle\int_{Q_T} \sgn\Big({V}^k(s,y)-{U}^k(t,x)\Big)\overline{R}^k (\boldsymbol{V},(\overline{\boldsymbol{\Upsilon}} \circledast {\boldsymbol{V}})^k)(s,y)\phi\d{y}\d{s}\leq 0.
\end{align}}
Integrating \eqref{611} and \eqref{6111} with respect to $y,s$ and $x,t$, respectively, and adding, we get:
{\allowdisplaybreaks\begin{align*}\label{31}
&-\displaystyle \int_{{\color{black}{\color{black}Q^2_T}}} \abs{{U}^k(t,x)-{V}^k(s,y)}(\phi_t+\phi_s)  \d{x} \d{t} \d{y} \d{s}\\
&-\displaystyle \int_{{\color{black}{\color{black}Q^2_T}}}\sgn \Big({U}^k(t,x)-{V}^k(s,y)\Big)\Big(f^k({U}^k(t,x)){\nu}^k((\boldsymbol{\Theta} \circledast  \boldsymbol{U})^k(t,x))\\
&\quad\quad\quad\quad\quad\quad\quad\quad\quad\quad\quad\quad-\overline{f}^k({V}^k(s,y))\overline{\nu}^k((\overline{\boldsymbol{\Theta}} \circledast  \boldsymbol{V})^k(s,y))\Big)(\phi_x+\phi_y) \d{x} \d{t} \d{y} \d{s}\\&
-\displaystyle \int_{{\color{black}{\color{black}Q^2_T}}} \sgn \Big({U}^k(t,x)-{V}^k(s,y)\Big) \Big(\Big(\left(\overline{f}^k({V}^k(s,y))\overline{\nu}^k((\overline{\boldsymbol{\Theta}} \circledast  \boldsymbol{V})^k(s,y))\right.\\
&\quad \quad \quad  \quad \quad\quad \quad \quad  \quad \quad\quad \quad \quad  \quad \quad \quad \quad\left.-f^k({V}^k(s,y)){\nu}^k((\boldsymbol{\Theta} \circledast  \boldsymbol{U})^k(t,x))\right)\phi\Big)_x\\&\quad \quad \quad \quad \quad \quad  \quad \quad  -\Big(\Big(f^k({U}^k(t,x)){\nu}^k((\boldsymbol{\Theta} \circledast  \boldsymbol{U})^k(t,x))\\&\quad \quad \quad \quad \quad \quad  \quad \quad \quad \quad \quad \quad\quad \quad-\overline{f}^k({U}^k(t,x))\overline{\nu}^k((\overline{\boldsymbol{\Theta}} \circledast  \boldsymbol{V})^k(s,y))\Big)\phi\Big)_y  \Big)\d{x} \d{t} \d{y} \d{s}\\
&-\displaystyle\int_{Q_T} \sgn\Big({U}^k(t,x)-{V}^k(s,y)\Big)\Big(R^k (\boldsymbol{U},(\boldsymbol{\Upsilon} \circledast {\boldsymbol{U}})^k)(t,x)\\
&\quad \quad \quad \quad \quad \quad\quad \quad \quad \quad \quad \quad\quad \quad \quad \quad -\overline{R}^k (\boldsymbol{V},(\overline{\boldsymbol{\Upsilon}} \circledast {\boldsymbol{V}})^k)(s,y)\Big)\phi \d{x} \d{t} \d{y} \d{s}
\leq 0.
\end{align*}}
Simplifying the above equations, we get:
{\allowdisplaybreaks\begin{align}
\label{eq:I}
- \displaystyle \int_{{\color{black}{\color{black}Q^2_T}}}(I_0+I_1+I_2+I_S)\, \d{x} \d{t} \d{y} \d{s} \leq  0,
\end{align}}
where
{\allowdisplaybreaks
\begin{align*}
I_0&=\abs{{U}^k(t,x)-{V}^k(s,y)}(\phi_t+\phi_s),\\
I_1&=\sgn \Big({U}^k(t,x)-{V}^k(s,y)\Big)\Big(f^k({U}^k(t,x)){\nu}^k((\boldsymbol{\Theta} \circledast  \boldsymbol{U})^k(t,x))\\
&\quad\quad\quad\quad\quad\quad\quad\quad\quad\quad\quad\quad\quad\quad-\overline{f}^k({V}^k(s,y))\overline{\nu}^k((\overline{\boldsymbol{\Theta}} \circledast  \boldsymbol{V})^k(s,y))\Big)(\phi_x+\phi_y),\\
I_2&= \sgn \Big({U}^k(t,x)-{V}^k(s,y)\Big) \Big(\Big(\left(\overline{f}^k({V}^k(s,y))\overline{\nu}^k((\overline{\boldsymbol{\Theta}} \circledast  \boldsymbol{V})^k(s,y))\right.\\
&\quad \quad \quad \quad \quad \quad \quad\quad \quad \quad \quad \quad \quad \quad\quad \quad \quad \quad \left.-f^k({V}^k(s,y)){\nu}^k((\boldsymbol{\Theta} \circledast  \boldsymbol{U})^k(t,x))\right)\phi\Big)_x\\
& \quad-\Big(\Big(f^k({U}^k(t,x)){\nu}^k((\boldsymbol{\Theta} \circledast  \boldsymbol{U})^k(t,x))-\overline{f}^k({U}^k(t,x))\overline{\nu}^k((\overline{\boldsymbol{\Theta}} \circledast  \boldsymbol{V})^k(s,y))\Big)\phi\Big)_y  \Big),\\
I_S&=\sgn\Big({U}^k(t,x)-{V}^k(s,y)\Big)\Big(R^k (\boldsymbol{U},(\boldsymbol{\Upsilon} \circledast {\boldsymbol{U}})^k)(t,x)-\overline{R}^k (\boldsymbol{V},(\overline{\boldsymbol{\Upsilon}} \circledast {\boldsymbol{V}})^k)(s,y)\Big)\phi.
\end{align*}}
Next, we introduce a non-negative function $\delta\in {{C_c^{\infty}}}(\R)$ such that
\[ \displaystyle \int_{\R}\delta({\omega})\d{{\omega}}=1,\ \delta({\omega})=\delta(-{\omega}),\ \delta({\omega})=0,\ \text{ for }|{\omega}|\ge 1,\] and set
    \[\xi_{\rho}(s):=\frac{1}{\rho}\delta\left(\frac{s}{\rho}\right),\quad \rho\in \R^+,\ s\in \R.\]
   We then choose $\phi:=\Phi=\Phi(t,x,s,y)\in {C_{{\color{black}c}}^{\infty}}({\color{black}Q^2_T}),$ by  \begin{align*}\Phi(t,x,s,y)=\psi(t,x)\xi_{\rho}(t-s)\xi_{\rho}(x-y),
\end{align*} where  $\psi(t,x)\in {C_c^{\infty}}({\color{black}Q_T})$ is a non-negative test function.
It is then straightforward to see that
\begin{align*}
&\left(\overline{f}^k({V}^k(s,y))\overline{\nu}^k((\overline{\boldsymbol{\Theta}} \circledast  \boldsymbol{V})^k(s,y))-f^k({V}^k(s,y)){\nu}^k((\boldsymbol{\Theta} \circledast  \boldsymbol{U})^k(t,x))\right)\Phi_x\\
&\quad= \left(\overline{f}^k({V}^k(s,y))\overline{\nu}^k((\overline{\boldsymbol{\Theta}} \circledast  \boldsymbol{V})^k(s,y))\right.\\
&\quad\quad\quad\quad\left.-f^k({V}^k(s,y)){\nu}^k((\boldsymbol{\Theta} \circledast  \boldsymbol{U})^k(t,x))\right)\psi(t,x)\xi_{\rho}(t-s)\xi^{'}_{\rho}(x-y)\\
&\qquad+\left(\overline{f}^k({V}^k(s,y))\overline{\nu}^k((\overline{\boldsymbol{\Theta}} \circledast  \boldsymbol{V})^k(s,y))\right.\\
&\quad\quad\quad\quad\left.-f^k({V}^k(s,y)){\nu}^k((\boldsymbol{\Theta} \circledast  \boldsymbol{U})^k(t,x))\right)\psi_x(t,x)\xi_{\rho}(t-s)\xi_{\rho}(x-y)
\end{align*}
and
\begin{align*}
&\Big(f^k({U}^k(t,x)){\nu}^k((\boldsymbol{\Theta} \circledast  \boldsymbol{U})^k(t,x))-\overline{f}^k({U}^k(t,x))\overline{\nu}^k((\overline{\boldsymbol{\Theta}} \circledast  \boldsymbol{V})^k(s,y))\Big)\Phi_y\\&\quad= \Big(-f^k({U}^k(t,x)){\nu}^k((\boldsymbol{\Theta} \circledast  \boldsymbol{U})^k(t,x))\\
&\quad\quad\quad\quad\quad\quad+\overline{f}^k({U}^k(t,x))\overline{\nu}^k((\overline{\boldsymbol{\Theta}} \circledast  \boldsymbol{V})^k(s,y))\Big)\psi(t,x)\xi_{\rho}(t-s)\xi^{'}_{\rho}(x-y).
\end{align*}
We write $I_2=I_{2,1}+I_{2,2}$, where $I_{2,1}$ has terms with derivatives of $\Phi$ and $I_{2,2}$ has terms without the derivatives of $\Phi$, i.e.,
\begin{align*}
I_2=&{\color{black}I_{2,1}}+{\color{black}I_{2,2}}\\
=&\sgn \Big({U}^k(t,x)-{V}^k(s,y)\Big) \Big(\Big(\overline{f}^k({V}^k(s,y))\overline{\nu}^k((\overline{\boldsymbol{\Theta}} \circledast  \boldsymbol{V})^k(s,y))\\
&\quad\quad\quad\quad\quad\quad\quad\quad\quad\quad\quad\quad\quad-f^k({V}^k(s,y)){\nu}^k((\boldsymbol{\Theta} \circledast  \boldsymbol{U})^k(t,x))\Big)\Phi_x\\ &{\quad -\Big(f^k({U}^k(t,x)){\nu}^k((\boldsymbol{\Theta} \circledast  \boldsymbol{U})^k(t,x))-\overline{f}^k({U}^k(t,x))\overline{\nu}^k((\overline{\boldsymbol{\Theta}} \circledast  \boldsymbol{V})^k(s,y))\Big)\Phi_y\Big)}\\
&\quad+{\color{black}\sgn \Big({U}^k(t,x)-{V}^k(s,y)\Big)\Phi \Big(-\partial_x({\nu}^k((\boldsymbol{\Theta} \circledast  \boldsymbol{U})^k(t,x)))f^k({V}^k(s,y))} \\
&\quad\quad\quad\quad\quad\quad\quad\quad \quad\quad\quad\quad\quad\quad\quad+\partial_y({\overline{\nu}}^k((\overline{\boldsymbol{\Theta}} \circledast  \boldsymbol{V})^k(s,y)))\overline{f}^k({U}^k(t,x))\Big).
\end{align*}
 We then write $I_{2,1}=I_{2,1,1}+I_{2,1,2}$, where $I_{2,1,1}$ has terms of only $\xi_{\rho}$ and $I_{2,1,2}$ has terms of $\xi'_{\rho}$
i.e.,
{\allowdisplaybreaks\begin{align*}
I_{2,1,1}&=\sgn \Big({U}^k(t,x)-{V}^k(s,y)\Big)\Big(\overline{f}^k({V}^k(s,y))\overline{\nu}^k((\overline{\boldsymbol{\Theta}} \circledast  \boldsymbol{V})^k(s,y))\\
&\quad\quad\quad\quad\quad\quad\quad\quad\quad\quad- f^k({V}^k(s,y)){\nu}^k((\boldsymbol{\Theta} \circledast  \boldsymbol{U})^k(t,x))\Big)\psi_x\xi_{\rho}(t-s)\xi_{\rho}(x-y),\\
I_{2,1,2}&=\sgn \Big({U}^k(t,x)-{V}^k(s,y)\Big)\left(\overline{f}^k({V}^k(s,y))\overline{\nu}^k((\overline{\boldsymbol{\Theta}} \circledast  \boldsymbol{V})^k(s,y))\right.\\
&\quad\quad\left.-f^k({V}^k(s,y)){\nu}^k((\boldsymbol{\Theta} \circledast  \boldsymbol{U})^k(t,x))+f^k({U}^k(t,x)){\nu}^k((\boldsymbol{\Theta} \circledast  \boldsymbol{U})^k(t,x))\right.\\&\left.\quad\quad-\overline{f}^k({U}^k(t,x))\overline{\nu}^k((\overline{\boldsymbol{\Theta}} \circledast  \boldsymbol{V})^k(s,y))\right) \psi(t,x)\xi_{\rho}(t-s) \xi^{'}_{\rho}(x-y)
\\
&=\Big(F^k({U}^k(t,x),{V}^k(s,y)){\nu}^k((\boldsymbol{\Theta} \circledast  \boldsymbol{U})^k(t,x))\\&\quad\quad-\overline{F}^k({U}^k(t,x),{V}^k(s,y))\overline{\nu}^k((\overline{\boldsymbol{\Theta}} \circledast  \boldsymbol{V})^k(s,y))\Big) \times \psi(t,x)\xi_{\rho}(t-s) \xi^{'}_{\rho}(x-y),
\end{align*}}
where $F^k(a,b)=\sgn (a-b)(f^k(a)-f^k(b)) \ \text{and}\ \overline{F}^k(a,b)=\sgn (a-b)(\overline{f}^k(a)-\overline{f}^k(b)).$ 
Applying integration by parts on $I_{2,1,2},$  we get rid of the derivatives of the term $\xi_{\rho}$
as below: \begin{align*}&\displaystyle\int_{{\color{black}Q^2_T}}I_{2,1,2}\d{x} \d{t} \d{y} \d{s}\\ 
&=\displaystyle \int_{{\color{black}Q^2_T}}\Big(\overline{F}^k_x({U}^k,{V}^k)\overline{\nu}^k((\overline{\boldsymbol{\Theta}} \circledast  \boldsymbol{V})^k(s,y))-F^k_x({U}^k,{V}^k){\nu}^k((\boldsymbol{\Theta} \circledast  \boldsymbol{U})^k(t,x))\\
&\quad\quad\quad-F^k({U}^k,{V}^k)\partial_x({\nu}^k((\boldsymbol{\Theta} \circledast  \boldsymbol{U})^k(t,x)))\Big)\psi(t,x)\xi_{\rho}(t-s)\xi_{\rho}(x-y)\d{x} \d{t} \d{y} \d{s}\\
&\quad\quad+ \displaystyle \int_{{\color{black}Q^2_T}}\Big(\overline{F}^k({U}^k,{V}^k)\overline{\nu}^k((\overline{\boldsymbol{\Theta}} \circledast  \boldsymbol{V})^k(s,y))\\ \nonumber
&\quad\quad\quad\quad\quad\quad\quad-F^k({U}^k,{V}^k){\nu}^k((\boldsymbol{\Theta} \circledast  \boldsymbol{U})^k(t,x))\Big)\psi_x\xi_{\rho}(t-s)\xi_{\rho}(x-y)\d{x} \d{t} \d{y} \d{s}.\\
\end{align*}
Now $I_{2,1,2}+I_{2,2}$ 
\begin{align*}
&=\Big(\overline{F}^k_x({U}^k,{V}^k)\overline{\nu}^k((\overline{\boldsymbol{\Theta}} \circledast  \boldsymbol{V})^k(s,y))-F^k_x({U}^k,{V}^k){\nu}^k((\boldsymbol{\Theta} \circledast  \boldsymbol{U})^k(t,x))\Big)\Phi\\
&\quad+\Big(\overline{F}^k({U}^k,{V}^k)\overline{\nu}^k((\overline{\boldsymbol{\Theta}} \circledast  \boldsymbol{V})^k(s,y))\\
&\quad\quad\quad\quad\quad\quad\quad\quad-F^k({U}^k,{V}^k){\nu}^k((\boldsymbol{\Theta} \circledast  \boldsymbol{U})^k(t,x))\Big)\psi_x\xi_{\rho}(t-s)\xi_{\rho}(x-y)\\
&\quad+{\color{black}\sgn \Big({U}^k(t,x)-{V}^k(s,y)\Big) \Big(\overline{f}^k({U}^k(t,x))\partial_y({\overline{\nu}}^k((\overline{\boldsymbol{\Theta}} \circledast  \boldsymbol{V})^k(s,y)))}\\ &\quad\quad\quad\quad\quad\quad\quad\quad\quad\quad\quad\quad\quad\quad\quad\quad-f^k({U}^k(t,x))\partial_x({\nu}^k((\boldsymbol{\Theta} \circledast  \boldsymbol{U})^k(t,x)))\Big)\Phi.
\end{align*}

Since ${U}^k(t,\cdot) \in \bv(\R),$ we have (see \cite[Lemma 4.1]{KR2001} for details): \begin{align*}
    \abs{F^k_x({U}^k,{V}^k)}\le {\color{black}\abs{f^k}_{\lipR}}\abs{\partial_x {U}^k},\,\,\text{and}\,\,\abs{\overline{F}^k_x({U}^k,{V}^k)}\le{\color{black}\abs{\overline{f}^k}_{\lipR}}\abs{\partial_x {U}^k}.
\end{align*} Also, note that 
\begin{align*}\allowdisplaybreaks
&\abs{\overline{F}^k_x({U}^k,{V}^k)\overline{\nu}^k((\overline{\boldsymbol{\Theta}} \circledast  \boldsymbol{V})^k(s,y))-F^k_x({U}^k,{V}^k){\nu}^k((\boldsymbol{\Theta} \circledast  \boldsymbol{U})^k(t,x))}\\
&\le\abs{\overline{F}^k_x({U}^k,{V}^k)-F^k_x({U}^k,{V}^k)}\abs{\overline{\nu}^k((\overline{\boldsymbol{\Theta}} \circledast  \boldsymbol{V})^k(s,y))}\\ \nonumber
&\quad\quad+\abs{F^k_x({U}^k,{V}^k)}\abs{\overline{\nu}^k((\overline{\boldsymbol{\Theta}} \circledast  \boldsymbol{V})^k(s,y))-{\nu}^k((\boldsymbol{\Theta} \circledast  \boldsymbol{U})^k(t,x))}\\
    &\le {\abs{\overline{f}^k-f^k}_{\lipR}}\abs{\partial_x {U}^k}\abs{\overline{\nu}^k((\overline{\boldsymbol{\Theta}} \circledast  \boldsymbol{V})^k(s,y))}\\ 
&\quad\quad+\abs{f^k}_{\lipR}\abs{\partial_x {U}^k}\abs{\overline{\nu}^k((\overline{\boldsymbol{\Theta}} \circledast  \boldsymbol{V})^k(s,y))-{\nu}^k((\boldsymbol{\Theta} \circledast  \boldsymbol{U})^k(t,x))}\end{align*}
    and
    \begin{align*}\allowdisplaybreaks
&\abs{\overline{f}^k({U}^k(t,x))\partial_y({\overline{\nu}}^k((\overline{\boldsymbol{\Theta}} \circledast  \boldsymbol{V})^k(s,y)))-f^k({U}^k(t,x))\partial_x({\nu}^k((\boldsymbol{\Theta} \circledast  \boldsymbol{U})^k(t,x)))}\\ 
&\le\abs{\overline{f}^k({U}^k(t,x))-f^k({U}^k(t,x))}\abs{\partial_y({\overline{\nu}}^k((\overline{\boldsymbol{\Theta}} \circledast  \boldsymbol{V})^k(s,y)))}\\
&\quad+\abs{f^k({U}^k(t,x))} \abs{\partial_y({\overline{\nu}}^k((\overline{\boldsymbol{\Theta}} \circledast  \boldsymbol{V})^k(s,y)))-\partial_x({\nu}^k((\boldsymbol{\Theta} \circledast  \boldsymbol{U})^k(t,x)))}\\
&\le \abs{\overline{f}^k-f^k}_{\lipR}\abs{\partial_y({\overline{\nu}}^k((\overline{\boldsymbol{\Theta}} \circledast  \boldsymbol{V})^k(s,y)))}\abs{{U}^k(t,x)}\\
&\quad+\abs{f^k}_{\lipR}\abs{{U}^k(t,x)}\abs{\partial_y({\overline{\nu}}^k((\overline{\boldsymbol{\Theta}} \circledast  \boldsymbol{V})^k(s,y)))-\partial_x({\nu}^k((\boldsymbol{\Theta} \circledast  \boldsymbol{U})^k(t,x)))}.
\end{align*}
Hence, we have, $I_{2,1,2}+I_{2,2}$
\begin{align*}
&\le \Big({\abs{\overline{f}^k-f^k}_{\lipR}}\abs{\partial_x {U}^k}\abs{\overline{\nu}^k((\overline{\boldsymbol{\Theta}} \circledast  \boldsymbol{V})^k(s,y))}\\ \nonumber
&\quad+\abs{f^k}_{\lipR}\abs{\partial_x {U}^k}\abs{\overline{\nu}^k((\overline{\boldsymbol{\Theta}} \circledast  \boldsymbol{V})^k(s,y))-{\nu}^k((\boldsymbol{\Theta} \circledast  \boldsymbol{U})^k(t,x))}\Big)\abs{\Phi}\\
&\quad+\Big|\overline{F}^k({U}^k,{V}^k)\overline{\nu}^k((\overline{\boldsymbol{\Theta}} \circledast  \boldsymbol{V})^k(s,y))\\
&\quad\quad\quad\quad\quad\quad\quad\quad-F^k({U}^k,{V}^k){\nu}^k((\boldsymbol{\Theta} \circledast  \boldsymbol{U})^k(t,x))\Big||\psi_x|\xi_{\rho}(t-s)\xi_{\rho}(x-y)\\
&\quad+\Big(\abs{\overline{f}^k-f^k}_{\lipR}\abs{\partial_y({\overline{\nu}}^k((\overline{\boldsymbol{\Theta}} \circledast  \boldsymbol{V})^k(s,y)))}\abs{{U}^k(t,x)}\\
&\quad\quad\quad+\abs{f^k}_{\lipR}\abs{{U}^k(t,x)}\abs{\partial_y({\overline{\nu}}^k((\overline{\boldsymbol{\Theta}} \circledast  \boldsymbol{V})^k(s,y)))-\partial_x({\nu}^k((\boldsymbol{\Theta} \circledast  \boldsymbol{U})^k(t,x)))}\Big)\abs{\Phi}.
\end{align*}
Now fix $0<t_1<t<t_2<T$ and choose
\begin{align*} \psi(t,x)&=\psi_{r,\Xi}(t,x)=\Psi^1_{r}(x)\Psi^2_{{\Xi}}(t),  \quad  \quad \quad  \, \,\, \quad \quad  r>1, {\Xi}>0,\\ 
\Psi^1_{r}(x)&=\int_{\R}\delta(|x-y|)\chi_{|y|<r}\d{y},   \quad \, \, \,  \quad \quad \quad  \quad \quad x\in \R, \\
\Psi^2_{{\Xi}}(t)&=\int_{-\infty}^t
\Big(\xi_{{\Xi}}(\tau(t,x)-t_1)-\xi_{{\Xi}}(\tau(t,x)-t_2)\Big)d\tau,    \quad \quad 0<t_1<t<t_2<T,
\end{align*}
so that the terms containing $\psi_x$ will go to zero when 
$r\uparrow \infty$. Hence, taking the limits $\rho \downarrow 0$ and $r\uparrow \infty$ in (\ref{eq:I}) effectively implies that,

\begin{align}\label{star}
&- \lim_{\rho \downarrow 0, r\uparrow \infty}\displaystyle \int_{{\color{black}Q^2_T} }(I_0+I_1+I_{2,1,1}+I_{2,1,2}+I_{2,2})\, \d{x} \d{t} \d{y} \d{s} \leq  0.
\end{align} Essentially, \allowdisplaybreaks we get,
\begin{align*}
&\lim_{\rho \downarrow 0, r\uparrow \infty}\int_{{\color{black}Q^2_T}}I_0 \d{x} \d{t} \d{y} \d{s}=\int_{{\color{black}Q_T}}|{U}^k(t,x)-{V}^k(t,x)|\Psi^{2'}_{{\Xi}}(t) \d{x} \d{t},\\
&\lim_{\rho \downarrow 0, r\uparrow \infty}\int_{{\color{black}Q^2_T}}I_1 \d{x} \d{t} \d{y} \d{s}=0,
\\
&\lim_{\rho \downarrow 0, r\uparrow \infty}\int_{{\color{black}Q^2_T}}I_{S} \d{x} \d{t} \d{y} \d{s}=\int_{{\color{black}Q_T}}\sgn\Big({U}^k(t,x)-{V}^k(t,x)\Big)\Big(R^k (\boldsymbol{U},(\boldsymbol{\Upsilon} \circledast {\boldsymbol{U}})^k)(t,x)\\
&\quad\quad\quad\quad\quad\quad\quad\quad\quad\quad\quad\quad\quad\quad\quad\quad\quad\quad\quad\quad-\overline{R}^k (\boldsymbol{V},(\overline{\boldsymbol{\Upsilon}} \circledast {\boldsymbol{V}})^k)(t,x)\Big)\Psi^2_{{\Xi}}(t)\d{x}\d{t},
\\
&\lim_{\rho \downarrow 0, r\uparrow \infty}\int_{{\color{black}Q^2_T}}I_{2,1,1} \d{x} \d{t} \d{y} \d{s}=0,
\\
&\lim_{\rho \downarrow 0, r\uparrow \infty}\int_{{\color{black}Q^2_T}}(I_{2,1,2}+I_{2,2}) \d{x} \d{t} \d{y} \d{s}\\
& \le \int_{{\color{black}Q_T}}\Big(\abs{\overline{f}^k-f^k}_{\lipR}\abs{\partial_x {U}^k}\abs{\overline{\nu}^k((\overline{\boldsymbol{\Theta}} \circledast  \boldsymbol{V})^k(t,x))}\\
&\quad\quad\quad\quad+\abs{f^k}_{\lipR}\abs{\partial_x{U}^k}\abs{\overline{\nu}^k((\overline{\boldsymbol{\Theta}} \circledast  \boldsymbol{V})^k(t,x))-{\nu}^k((\boldsymbol{\Theta} \circledast  \boldsymbol{U})^k(t,x))}\Big)\Psi^2_{{\Xi}}(t) \d{x} \d{t}\\
&\quad+\int_{{\color{black}Q_T}}\Big( \abs{\overline{f}^k-f^k}_{\lipR}\abs{\partial_x({\overline{\nu}}^k((\overline{\boldsymbol{\Theta}} \circledast  \boldsymbol{V})^k(t,x)))}\abs{{U}^k(t,x)}\\
&\quad\quad\quad\quad\quad\quad+\abs{f^k}_{\lipR}\abs{{U}^k(t,x)}\abs{\partial_x({\overline{\nu}}^k((\overline{\boldsymbol{\Theta}} \circledast  \boldsymbol{V})^k(t,x)))\\
&\quad\quad\quad\quad\quad\quad\quad\quad\quad\quad\quad\quad-\partial_x({\nu}^k((\boldsymbol{\Theta} \circledast  \boldsymbol{U})^k(t,x)))}\Big)\Psi^2_{{\Xi}}(t) \d{x} \d{t}.
\end{align*} 
Further as $\Xi\downarrow 0,$ \eqref{star} implies 
{\allowdisplaybreaks
\begin{align}\label{eq:ut2vt2ut1vt1}
   & \displaystyle \int_{\R}{\color{black}\abs{{{U}^k}(t_2,x)-{{V}^k}(t_2,x)}}\d{x}-\displaystyle \int_{\R}{\color{black}\abs{{{U}^k}(t_1,x)-{{V}^k}(t_1,x)}}\d{x}\\ \nonumber
 &\leq  \abs{\overline{f}^k-f^k}_{\lipR}\displaystyle \int_{t_1}^{t_2}\displaystyle \int_{\R} \Big(\abs{\partial_x {U}^k}\abs{\overline{\nu}^k((\overline{\boldsymbol{\Theta}}\circledast \boldsymbol{V})^k(t,x))}\\
&\quad\quad\quad\quad\quad\quad\quad\quad\quad+\abs{\partial_x(\overline{\nu}^k((\overline{\boldsymbol{\Theta}}\circledast \boldsymbol{V})^k(t,x)))}\abs{{U}^k(t,x)}\Big)\d{x}\d{t}\\ \nonumber
&\quad+\abs{f^k}_{\lipR}\int_{t_1}^{t_2}\displaystyle \int_{\R}\Big(\abs{\partial_x {U}^k}\abs{\overline{\nu}^k((\overline{\boldsymbol{\Theta}}\circledast \boldsymbol{V})^k(t,x))-\nu^k((\boldsymbol{\Theta}\circledast \boldsymbol{U})^k(t,x))}\\ \nonumber
&\quad+\abs{{U}^k(t,x)}\abs{\partial_x(\overline{\nu}^k((\overline{\boldsymbol{\Theta}}\circledast \boldsymbol{V})^k(t,x)))-\partial_x(\nu^k((\boldsymbol{\Theta}\circledast \boldsymbol{U})^k(t,x)))}\Big)\d{x}\d{t}\\ \nonumber
&\quad+\int_{t_1}^{t_2}\displaystyle \int_{\R}\abs{R^k (\boldsymbol{U},(\boldsymbol{\Upsilon} \circledast {\boldsymbol{U}})^k)(t,x)-\overline{R}^k (\boldsymbol{V},(\overline{\boldsymbol{\Upsilon}} \circledast {\boldsymbol{V}})^k)(t,x)}\d{x}\d{t}\\ \nonumber
&:=J_1+J_2+J_3. \nonumber
\end{align}}
To prove the theorem, we  estimate the right hand side of (\ref{eq:ut2vt2ut1vt1}). To this end, we first observe that:
{\allowdisplaybreaks
\begin{align*}
\left|\overline{\nu}^k((\overline{\boldsymbol{\Theta}}\circledast \boldsymbol{V})^k(t,x))\right|&\le \norma{{\overline{\nu}^k}}_{L^{\infty}(\mathbb{R}^N)}, \\
\norma{\partial_x((\overline{\boldsymbol{\Theta}}\circledast \boldsymbol{V})^k(t,x))}&=\sum_{j\in\mathcal{N}}\left|\partial_x( (\overline{\Theta}^{j,k}*V^j)(t,x))\right|\\
&=\sum_{j\in\mathcal{N}}\left| \displaystyle\int_0^t\int_{\R} V^j(\tau,\xi)\dot{\overline{\mu}}^{j,k}(x-\xi)\overline{\Gamma}^{j,k}(t-\tau) \d \xi \d \tau\right|\\
&\le\sum_{j\in\mathcal{N}}\norma{\dot{\overline{\mu}}^{j,k}}_{L^\infty(\R)}\norma{\overline{\Gamma}^{j,k}}_{L^\infty(\R^{+})}\norma{ V^j}_{L^1(Q_T)}\\
&\le\norma{\overline{\boldsymbol\mu}'}_{(L^\infty(\R))^{N^2}}\norma{\overline{\boldsymbol\Gamma}}_{(L^\infty(\R^{+}))^{N^2}}\norma{ \boldsymbol{V}}_{(L^1(Q_T))^N}\\
&:=\mathcal{C}_{11},
\end{align*}}
\vspace{-5mm}
\begin{align*}
&\abs{\partial_x(\overline{\nu}^k((\overline{\boldsymbol{\Theta}}\circledast \boldsymbol{V})^k(t,x)))}\\
&\le\sum\limits_{j\in\mathcal{N}}\abs{\partial_{j}\nu^{k}\!\left((\overline{\boldsymbol{\Theta}}\circledast \boldsymbol{V})^k(t,x)\right)}\,
\int_{0}^{t}\!\int_{\mathbb{R}}
\abs{V^{j}(\tau,\xi)}\,\abs{\dot{\overline{\mu}}^{j,k}(x-\xi)}\abs{\overline{\Gamma}^{j,k}(t-\tau)}  \d \xi \d \tau\\
&\le\norma{{\color{black}\nabla }\overline{\boldsymbol{\nu}}}_{(L^{\infty}(\R^N))^{N^2}}\norma{\overline{\boldsymbol{\Gamma}}}_{(L^{\infty}(\R^{+}))^{N^2}}\norma{\overline{\boldsymbol{\mu}}'}_{(L^{\infty}(\R))^{N^2}}\norma{\boldsymbol{V}}_{(L^1(Q_T))^N}\\
&=:{\mathcal{C}_{12}},
\end{align*}
which implies
{\begin{align}
J_1 
&\le\abs{\overline{f}^k-f^k}_{\lipR}\displaystyle \int_{t_1}^{t_2}\displaystyle \int_{\R}\Big(\abs{\partial_x {U}^k} \norma{{\overline{\nu}^k}}_{L^{\infty}(\mathbb{R}^N)}+{\mathcal{C}_{12}}\abs{{U}^k(t,x)}\Big) \d{x}\d{t}
\\
    & \le \abs{\overline{f}^k-f^k}_{\lipR}\Big(T \norma{{\overline{\nu}^k}}_{L^{\infty}(\mathbb{R}^N)}|{U}^k|_{L^\infty_t\bv_x}+{\mathcal{C}_{12}}\norma{U^k}_{L^{1}(Q_T)}\Big)\nonumber\\
         & \le{{\mathcal{C}_{21}}\abs{\overline{f}^k-f^k}_{\lipR} },\nonumber  
     \end{align}}
        where ${\mathcal{C}_{21}}:=T \norma{{\overline{\boldsymbol{\nu}}}}_{(L^{\infty}(\mathbb{R}^N))^N}|\boldsymbol{U}|_{(L^\infty_t\bv_x)^N}+{\mathcal{C}_{12}}\norma{\boldsymbol{U}}_{(L^{1}(Q_T))^N}.$\\
        
Furthermore, to estimate $J_2$, note that 
\begin{align} \nonumber
&\norma{(\overline{\boldsymbol{\Theta}}\circledast \boldsymbol{V})^k(t,x)-(\boldsymbol{\Theta}\circledast \boldsymbol{U})^k(t,x)}
\\ \nonumber &=\sum\limits_{j\in\mathcal{N}}\left| (\overline{\Theta}^{j,k}*V^j)(t,x) - (\Theta^{j,k}*U^j)(t,x)\right|\\ \nonumber
    &\le \sum\limits_{j\in\mathcal{N}}\left| (\overline{\Theta}^{j,k}*(U^j-V^j))(t,x)\right|+\left|((\overline{\Theta}^{j,k}-\Theta^{j,k})*U^j)(t,x)\right|\\ \nonumber
&\le\sum\limits_{j\in\mathcal{N}}\int_0^t\int_{\R}
\abs{U^j(\tau,\xi)-V^j(\tau,\xi)}\abs{\overline{\Theta}^{j,k}(t-\tau,x-\xi)} \d \xi \d \tau\\ \nonumber
&\quad+\sum\limits_{j\in\mathcal{N}}\int_0^t\int_{\R}
\abs{U^j(\tau,\xi)}\abs{(\overline{\Theta}^{j,k}-\Theta^{j,k})(t-\tau,x-\xi)} \d \xi \d \tau\\
\label{eq:Diffconv}
&\le\norma{\overline{\boldsymbol{\Theta}}}_{(L^{\infty}(\overline{Q}_T))^{N^2}}\int_0^t\norma{\boldsymbol{U}(\tau,\cdot)-\boldsymbol{V}(\tau,\cdot)}_{(L^1(\R))^N} \d \tau\\ \nonumber
&\quad+\norma{\overline{\boldsymbol{\Gamma}}}_{{(L^{\infty}(\R^{+}))}^{N^2}}\norma{\overline{\boldsymbol{\mu}}-\boldsymbol{\mu}}_{{(L^{\infty}(\R))}^{N^2}}\norma{\boldsymbol{U}}_{(L^1({Q}_T))^N}\\ \nonumber
&\quad+\norma{\boldsymbol{\mu}}_{{(L^{\infty}(\R))}^{N^2}}\norma{\overline{\boldsymbol{\Gamma}}-\boldsymbol{\Gamma}}_{{(L^{\infty}(\R^{+}))}^{N^2}}\norma{\boldsymbol{U}}_{(L^1({Q}_T))^N}, \nonumber
\end{align}
and 
\begin{align*}
&\norma{\partial_x((\overline{\boldsymbol{\Theta}}\circledast \boldsymbol{V})^k(t,x))-\partial_x((\boldsymbol{\Theta}\circledast \boldsymbol{U})^k(t,x))}\\
    &\quad \le\norma{ (\partial_x\overline{\boldsymbol{\Theta}}\circledast(\boldsymbol{V}- \boldsymbol{U}))^k(t,x)}+\norma{(\partial_x(\overline{\boldsymbol{\Theta}}-\boldsymbol{\Theta})\circledast \boldsymbol{U})^k(t,x)}\\
&\quad \le\sum\limits_{j\in\mathcal{N}}
\int_{0}^{t}\!\int_{\mathbb{R}}
\abs{(V^{j}-U^{j})(\tau,\xi)\,\dot{\overline{\mu}}^{j,k}(x-\xi)\overline{\Gamma}^{j,k}(t-\tau)}  \d \xi \d \tau\\
&\quad +\int_{0}^{t}\!\int_{\mathbb{R}}
\ \abs{U^{j}(\tau,\xi)\,\Big(\dot{\overline{\mu}}^{j,k}(x-\xi)\overline{\Gamma}^{j,k}(t-\tau)-\dot{\mu}^{j,k}(x-\xi)\Gamma^{j,k}(t-\tau)\Big)} \d \xi \d \tau\\
&\quad \le \norma{\boldsymbol{\mu}'}_{(L^{\infty}(\R))^{N^2}}\norma{\boldsymbol{\Gamma}}_{(L^{\infty}(\R^+))^{N^2}}\int_0^t\norma{\boldsymbol{U}(\tau,\cdot)-\boldsymbol{V}(\tau,\cdot)}_{(L^1(\R))^N} \d \tau\\
&\quad + \Big(\norma{\boldsymbol{\mu}'}_{(L^{\infty}(\R))^{N^2}}\norma{\overline{\boldsymbol{\Gamma}}-\boldsymbol{\Gamma}}_{(L^{\infty}(\R^+))^{N^2}}\\
&\quad\quad\quad\quad\quad\quad\quad\quad\quad+\norma{\overline{\boldsymbol{\mu}}'-\boldsymbol{\mu}'}_{(L^{\infty}(\R))^{N^2}}\norma{\boldsymbol{\Gamma}}_{(L^{\infty}(\R^+))^{N^2}}\Big) \norma{\boldsymbol{U}}_{(L^1(Q_T))^N}.
\end{align*}
Consequently,
 \begin{align*}
&\abs{\overline{\nu}^k((\overline{\boldsymbol{\Theta}}\circledast \boldsymbol{V})^k(t,x))-\nu^k((\boldsymbol{\Theta}\circledast \boldsymbol{U})^k(t,x))}\\
&\quad\le\abs{\overline{\nu}^k((\overline{\boldsymbol{\Theta}}\circledast \boldsymbol{V})^k(t,x))-\nu^k((\overline{\boldsymbol{\Theta}}\circledast\boldsymbol{V})^k(t,x))}\\
&\quad\quad+\abs{\nu^k((\overline{\boldsymbol{\Theta}}\circledast \boldsymbol{V})^k(t,x))-\nu^k((\boldsymbol{\Theta}\circledast \boldsymbol{U})^k(t,x))}\\
&\quad\le\norma{\overline{\nu}^k-\nu^k}_{L^{\infty}(\R^N)}+\abs{\nu^k}_{\lip(\R^N)}\norma{(\overline{\boldsymbol{\Theta}}\circledast \boldsymbol{V})^k(t,x)-(\boldsymbol{\Theta}\circledast \boldsymbol{U})^k(t,x)}\\
&\quad\le\norma{\overline{\nu}^k-\nu^k}_{L^{\infty}(\R^N)}+\mathcal{C}_{13}\int_0^t\norma{\boldsymbol{U}(\tau,\cdot)-\boldsymbol{V}(\tau,\cdot)}_{(L^1(\R))^N} \d \tau\\
&\quad\quad\quad+\mathcal{C}_{14}\norma{\overline{\boldsymbol{\mu}}-\boldsymbol{\mu}}_{{(L^{\infty}(\R))}^{N^2}}+\mathcal{C}_{15}\norma{\overline{\boldsymbol{\Gamma}}-\boldsymbol{\Gamma}}_{{(L^{\infty}(\R^{+}))}^{N^2}},
\end{align*}
 \vspace{-5mm}
\begin{align*}
\text{where\,}\quad \mathcal{C}_{13}&:=\abs{\boldsymbol{\nu}}_{(\lip(\R^N))^N}\norma{\overline{\boldsymbol{\Theta}}}_{{(L^{\infty}(\overline{Q}_T))}^{N^2}},\\
\mathcal{C}_{14}&:=\abs{\boldsymbol{\nu}}_{(\lip(\R^N))^N}\norma{\overline{\boldsymbol{\Gamma}}}_{{(L^{\infty}(\R^{+}))}^{N^2}}\norma{\boldsymbol{U}}_{(L^1(Q_T))^N},\\
\mathcal{C}_{15}&:=\abs{\boldsymbol{\nu}}_{(\lip(\R^N))^N}\norma{\boldsymbol{\mu}}_{{(L^{\infty}(\R))}^{N^2}}\norma{\boldsymbol{U}}_{(L^1(Q_T))^N},
\end{align*}
and
\begin{align*}
&\abs{\partial_x(\overline{\nu}^k((\overline{\boldsymbol{\Theta}}\circledast \boldsymbol{V})^k(t,x)))-\partial_x(\nu^k((\boldsymbol{\Theta}\circledast \boldsymbol{U})^k(t,x)))}\\
&\le{\mathcal{C}_{11}}\norma{\nabla\overline{\nu}^k\left((\overline{\boldsymbol{\Theta}}\circledast \boldsymbol{V})^k(t,x)\right)-\nabla\overline{\nu}^k\left((\boldsymbol{\Theta}\circledast \boldsymbol{U})^k(t,x)\right)}\\
&\quad+{\mathcal{C}_{11}}\norma{\nabla\overline{\nu}^k\left((\boldsymbol{\Theta}\circledast \boldsymbol{U})^k(t,x)\right)-\nabla{\nu}^k\left((\boldsymbol{\Theta}\circledast \boldsymbol{U})^k(t,x)\right)}\\
&\quad+\norma{\nabla{\nu}^k\left((\boldsymbol{\Theta}\circledast \boldsymbol{U})^k(t,x)\right)}\norma{\partial_x((\overline{\boldsymbol{\Theta}}\circledast \boldsymbol{V})^k(t,x))-\partial_x((\boldsymbol{\Theta}\circledast \boldsymbol{U})^k(t,x))}\\
&\le \mathcal{C}_{16}\int_0^t\norma{\boldsymbol{U}(\tau,\cdot)-\boldsymbol{V}(\tau,\cdot)}_{(L^1(\R))^N} \d \tau+{\mathcal{C}_{11}}\norma{\nabla \overline{\nu}^k-\nabla {\nu}^k}_{(L^{\infty}(\mathbb{R}^N))^N}\\ \nonumber
&\quad+\mathcal{C}_{17}\norma{\overline{\boldsymbol{\Gamma}}-\boldsymbol{\Gamma}}_{{(L^{\infty}(\R^{+}))}^{N^2}}+\mathcal{C}_{18}\norma{\overline{\boldsymbol{\mu}}-\boldsymbol{\mu}}_{{(W^{1,\infty}(\R))}^{N^2}},
 \end{align*}
\begin{align*}
 \text{where} \quad \mathcal{C}_{16}&:={\mathcal{C}_{11}}\norma{\operatorname{Hess} \overline{\boldsymbol{\nu}}}_{(L^{\infty}({\R}^N))^{N^3}}\norma{\overline{\boldsymbol{\Theta}}}_{{(L^{\infty}(\overline{Q}_T))}^{N^2}}\\
 &\quad\quad+\norma{\nabla {\boldsymbol{\nu}}}_{(L^{\infty}(\mathbb{R}^N))^{N^2}} \norma{\boldsymbol{\mu}'}_{(L^{\infty}(\R))^{N^2}}\norma{\boldsymbol{\Gamma}}_{(L^{\infty}(\R^+))^{N^2}},\\ 
\mathcal{C}_{17}&:={\mathcal{C}_{11}}\norma{\operatorname{Hess} \overline{\boldsymbol{\nu}}}_{(L^{\infty}({\R}^N))^{N^3}}\norma{\boldsymbol{\mu}}_{{(L^{\infty}(\R))}^{N^2}}\norma{\boldsymbol{U}}_{(L^1({Q}_T))^N}\\
&\quad\quad+\norma{\nabla {\nu}^k}_{(L^{\infty}(\mathbb{R}^N))^N}\norma{\boldsymbol{U}}_{(L^1(Q_T))^N}\norma{\boldsymbol{\mu}'}_{(L^{\infty}(\R))^{N^2}},\\
\mathcal{C}_{18}&:={\mathcal{C}_{11}}\norma{\operatorname{Hess} \overline{\boldsymbol{\nu}}}_{(L^{\infty}({\R}^N))^{N^3}}\norma{\overline{\boldsymbol{\Gamma}}}_{{(L^{\infty}(\R^{+}))}^{N^2}}\norma{\boldsymbol{U}}_{(L^1({Q}_T))^N}\\
&\quad\quad+\norma{\nabla {\nu}^k}_{(L^{\infty}(\mathbb{R}^N))^N}\norma{\boldsymbol{U}}_{(L^1(Q_T))^N}\norma{\boldsymbol{\Gamma}}_{(L^{\infty}(\R^+))^{N^2}},
\end{align*}
 which implies 
\begin{align}\label{J2}
J_2&\le\abs{f^k}_{\lipR}\int_{t_1}^{t_2}\displaystyle \int_{\R}\Big(\abs{\partial_x {U}^k}\abs{\overline{\nu}^k((\overline{\boldsymbol{\Theta}}\circledast \boldsymbol{V})^k(t,x))-\nu^k((\boldsymbol{\Theta}\circledast \boldsymbol{U})^k(t,x))}\\ \nonumber
&\quad+\abs{{U}^k(t,x)}\norma{\partial_x(\overline{\nu}^k((\overline{\boldsymbol{\Theta}}\circledast \boldsymbol{V})^k(t,x)))-\partial_x(\nu^k((\boldsymbol{\Theta}\circledast \boldsymbol{U})^k(t,x)))}\Big)\d{x}\d{t}\\ \nonumber
&\le {\mathcal{C}_{22}}\norma{\overline{\nu}^k-\nu^k}_{L^{\infty}(\R^N)}+{\mathcal{C}_{23}}\norma{\nabla \overline{\nu}^k-\nabla {\nu}^k}_{(L^{\infty}(\mathbb{R}^N))^N}+{\mathcal{C}_{24}}\norma{\overline{\boldsymbol{\Gamma}}-\boldsymbol{\Gamma}}_{(L^{\infty}(\R^+))^{N^2}}\\ \nonumber
&\quad+{\mathcal{C}_{25}}\norma{\overline{\boldsymbol{\mu}}-\boldsymbol{\mu}}_{(W^{1,\infty}(\R))^{N^2}}+{\mathcal{C}_{26}}\int_0^T\norma{\boldsymbol{U}(\tau,\cdot)-\boldsymbol{V}(\tau,\cdot)}_{(L^1(\R))^N} \d \tau, \nonumber
\end{align}
 where
\begin{align*}
 &{\mathcal{C}_{22}}:=T\abs{\boldsymbol{f}}_{(\lipR)^N}\norma{\boldsymbol{U}}_{(L_t^{\infty}\bv_x)^N},\\
 &{\mathcal{C}_{23}}:={\mathcal{C}_{11}}\abs{\boldsymbol{f}}_{(\lipR)^N}\norma{\boldsymbol{U}}_{(L^{1}(Q_T))^N},\\
 &{\mathcal{C}_{24}}:=T\mathcal{C}_{15}\abs{\boldsymbol{f}}_{(\lipR)^N}\norma{\boldsymbol{U}}_{(L_t^{\infty}\bv_x)^N}+\mathcal{C}_{17}\abs{\boldsymbol{f}}_{(\lipR)^N}\norma{\boldsymbol{U}}_{(L^{1}(Q_T))^N},\\
 &{\mathcal{C}_{25}}:=T\mathcal{C}_{14}\abs{\boldsymbol{f}}_{(\lipR)^N}\norma{\boldsymbol{U}}_{(L_t^{\infty}\bv_x)^N}+\mathcal{C}_{18}\abs{\boldsymbol{f}}_{(\lipR)^N}\norma{\boldsymbol{U}}_{(L^{1}(Q_T))^N},\\
  &{\mathcal{C}_{26}}:=T\mathcal{C}_{13}\abs{\boldsymbol{f}}_{(\lipR)^N}\norma{\boldsymbol{U}}_{(L_t^{\infty}\bv_x)^N}+\mathcal{C}_{16}\abs{\boldsymbol{f}}_{(\lipR)^N}\norma{\boldsymbol{U}}_{(L^{1}(Q_T))^N}.
 \end{align*}
Finally, using the Young's convolution inequality, we get:
{\allowdisplaybreaks\begin{align}\label{J3}
J_3&=\int_{t_1}^{t_2}\displaystyle \int_{\R}\abs{R^k (\boldsymbol{U},(\boldsymbol{\Upsilon} \circledast {\boldsymbol{U}})^k)(t,x)-\overline{R}^k (\boldsymbol{V},(\overline{\boldsymbol{\Upsilon}} \circledast {\boldsymbol{V}})^k)(t,x)}\d{x}\d{t}\\  \nonumber
&\le\int_{t_1}^{t_2}\displaystyle \int_{\R}\Big(\abs{R^k (\boldsymbol{U},(\boldsymbol{\Upsilon} \circledast {\boldsymbol{U}})^k)(t,x)-{R}^k (\boldsymbol{V},(\overline{\boldsymbol{\Upsilon}} \circledast {\boldsymbol{V}})^k)(t,x)}\\  \nonumber
&\quad\quad+\abs{{R}^k (\boldsymbol{V},(\overline{\boldsymbol{\Upsilon}} \circledast {\boldsymbol{V}})^k)(t,x)-\overline{R}^k (\boldsymbol{V},(\overline{\boldsymbol{\Upsilon}} \circledast {\boldsymbol{V}})^k)(t,x)}\Big)\d{x}\d{t}\\ \nonumber
&\le\int_{t_1}^{t_2}\displaystyle \int_{\R}\Big[\abs{R^k}_{\operatorname{Lip}({\R^{2N}})}\Big(\norma{ \boldsymbol{U}(t,x)-\boldsymbol{V}(t,x)}\\
&\quad\quad\quad\quad+\norma{(\boldsymbol{\Upsilon} \circledast {\boldsymbol{U}})^k(t,x)-(\overline{\boldsymbol{\Upsilon}} \circledast {\boldsymbol{V}})^k(t,x)}\Big)\\  \nonumber
&\quad\quad+\abs{{R}^k-\overline{R}^k}_{\operatorname{Lip}({\R^{2N}})} \Big(\norma{\boldsymbol{V}(t,x)}+\norma{(\overline{\boldsymbol{\Upsilon}} \circledast {\boldsymbol{V}})^k(t,x)}\Big)\Big]\d{x}\d{t}\\ \nonumber
&\le {\mathcal{C}_{31}}\int_{t_1}^{t_2}\norma{ \boldsymbol{U}(t,\cdot)-\boldsymbol{V}(t,\cdot)}_{(L^1(\R))^N}\d{t}+{\mathcal{C}_{32}}\norma{\overline{\boldsymbol{\eta}}-\boldsymbol{\eta}}_{(L^1(\R))^{N^2}}\\  \nonumber
&\quad\quad+{\mathcal{C}_{33}}\norma{\overline{\boldsymbol{\theta}}-\boldsymbol{\theta}}_{(L^1(\R^+))^{N^2}}+{\mathcal{C}_{34}}\abs{\boldsymbol{R}-\overline{\boldsymbol{R}}}_{(\operatorname{Lip}({\R^{2N}}))^N}, \nonumber
\end{align}}
where
{\allowdisplaybreaks\begin{align*}
{\mathcal{C}_{31}}&:=\abs{\boldsymbol{R}}_{(\operatorname{Lip}({\R^{2N}}))^N}\Big(1+\norma{\boldsymbol{\Upsilon}}_{(L^1(Q_T))^{N^2}}\Big), \\
    {\mathcal{C}_{32}}&:=\abs{\boldsymbol{R}}_{(\operatorname{Lip}({\R^{2N}}))^N}\norma{\boldsymbol{\theta}}_{(L^\infty(\R^+))^{N^2}}\norma{{\boldsymbol{V}}}_{(L^1(Q_T))^{N}},  \\
    {\mathcal{C}_{33}}&:=\abs{\boldsymbol{R}}_{(\operatorname{Lip}({\R^{2N}}))^N}\norma{\overline{\boldsymbol{\eta}}}_{(L^\infty(\R))^{N^2}}\norma{{\boldsymbol{V}}}_{(L^1(Q_T))^{N}},  \\
    {\mathcal{C}_{34}}&:= \Big(1+\norma{\overline{\boldsymbol{\Upsilon}}}_{(L^1(Q_T))^{N^2}}\Big)\norma{ {\boldsymbol{V}}}_{(L^1(Q_T))^N}.
\end{align*}}
Using \eqref{eq:ut2vt2ut1vt1}-\eqref{J3}, invoking the  time continuity of the entropy solution $U^k$ and $V^k$  
and summing over $k\in\mathcal{N},$ we have
the result with
$\mathcal{C}_{1}=N\mathcal{C}_{21},\ \mathcal{C}_{2}:=N\mathcal{C}_{22},\ \mathcal{C}_{3}:=N\mathcal{C}_{23}, \mathcal{C}_{3}:=N\mathcal{C}_{23},\ \mathcal{C}_{4}:=N\mathcal{C}_{24},\ \mathcal{C}_{5}:=N\mathcal{C}_{25},\ \mathcal{C}_{6}:=N {\mathcal{C}_{32}},\ 
\mathcal{C}_{7}:=N {\mathcal{C}_{33}},\ \mathcal{C}_{8}:=N{\mathcal{C}_{34}},\ 
\mathcal{C}_{9}:=N({\mathcal{C}_{26}}+{\mathcal{C}_{31}}).$
Applying Gronwall's inequality yields the result.
\end{proof}

\begin{remark}
    Theorem \ref{theorem:stability} can be extended to less regular spatial kernels such as one-sided kernels, ${\mu}^{j,k}\in (C^1 \cap W^{1,\infty})([0,\zeta_\mu])$ and $\ {\eta}^{j,k}\in (C^1 \cap W^{1,\infty})([0,\zeta_\eta])$, by essentially modifying the above estimates in the spatial convolution terms. We refer the readers to the nonlinearity paper \cite{AHV2024} for the necessary estimates on one sided kernels.  
\end{remark}

\begin{remark}
    The above theorem implies that the entropy solutions to a given IVP are unique. 
\end{remark}
\section{Numerical approximations}\label{num}
For $\Delta x, \Delta t>0,$ and $\lambda:=\Delta t/\Delta x,$ consider equidistant spatial grid points $x_i:=i\Delta x$ for $i\in \Z,$ and let $\chi_i(x)$ denote the indicator function of $C_i:=[x_{i-1/2}, x_{i+1/2})$, where $x_{i+1/2}=\frac{1}{2}(x_i+x_{i+1})$. Further, let $t^n:=n\Delta t$ 
for integers in $\mathcal{N}_T:=\{0, \ldots, N_T\}$, such that $T=N_T \D t$ denote the temporal grid points, and let
$\chi^n(t)$ denote the indicator function of $C^{n}:=[t^n,t^{n+1})$. For every $k\in\mathcal{N}$, we approximate the initial data \eqref{init}, according to:
\begin{equation*}
U^{k,\Delta}_0(x):=\sum\limits_{i\in\Z}\chi_i(x)U^{k,0}_i\quad \mbox{where }U^{k,0}_i=\int_{C_i}U_{0}^k(x)\d x, \quad i\in \Z,
\end{equation*}
and define a piecewise constant approximate solution $U^{k,\Delta}$ to~\eqref{nlm}-\eqref{init}
by:
$$
  U^{k,\Delta} (t,x) =  U^{k,n}_{i}
  \mbox{ for } 
(t,x)\in {C}^{n}\times C_i, (i,n)\in \Z\times\mathcal{N}_T. 
$$
For every $(i,k,n)\in \Z\times\mathcal{N}\times\mathcal{N}_T$, $U^{k,n}_{i}$ is defined via the following marching formula:
\begin{align}
\begin{split}U^{k,n}_i&=H^k(\nu^k(\boldsymbol{c}^{k,n-1}_{i-1/2}),\nu^k(\boldsymbol{c}^{k,n-1}_{i+1/2}),U_{i-1}^{k,n},U_i^{k,n-1},U_{i+1}^{k,n-1}) \\&\quad + \D t R^k(U_1^{k,n-1},,\ldots,U_N^{k,n-1},\boldsymbol{d}_i^{k,n-1})
    \\  
    & := 
   U^{k,n-1}_i- \lambda \big[
    \mathcal{F}^k(\nu^k(\boldsymbol{c}^{k,n-1}_{i+1/2}),U_i^{k,n-1},U_{i+1}^{k,n})
    - 
    \mathcal{F}^k(\nu^k(\boldsymbol{c}^{k,n-1}_{i-1/2}),U_{i-1}^{k,n-1},U_{i}^{k,n})
     \big]\\&\quad + \D t R^k(U_1^{k,n-1},,\ldots,U_N^{k,n-1},\boldsymbol{d}_i^{k,n-1})\\ 
     \label{scheme2}
     &:=U^{k,n-1}_i- \lambda \bigl[
    \mathcal{F}^{k,n}_{i+1/2} (U_i^{k,n-1},U_{i+1}^{k,n-1})
    - 
\mathcal{F}^{k,n-1}_{i-1/2} (U_{i-1}^{k,n-1},U_{i}^{k,n})
     \bigr]
     \\&\quad + \D t R^k(U_1^{k,n-1},,\ldots,U_N^{k,n-1},\boldsymbol{d}_i^{k,n-1}).
     \end{split}\end{align}
     Here,  $\boldsymbol{c}_{i+1/2}^{k,n}:= \left(c_{i+1/2}^{s,k,n}\right)_{s\in \mathcal{N}},\ \boldsymbol{d}_i^{k,n}=\left(d_{i+1/2}^{s,k,n}\right)_{s\in \mathcal{N}}$ and 
$\mathcal{F}^k(\nu^k(\boldsymbol{c}^{k,n}_{i+1/2}),U_i^{k,n},U_{i+1}^{k,n})
    $ denotes the numerical approximation of the flux $f^k(U^k)\nu^k((\boldsymbol{\Theta}\circledast  \boldsymbol{U})^k)$ at $(t^n,x_{i+1/2})$ for $(k,i,n)\in \mathcal{N}\times\Z\times\mathcal{N}_T$, where for every $s,k\in \mathcal{N}$, 
\begin{align*}c_{i+1/2}^{s,k,n}:=\Delta x \D t \sum\limits_{m=0}^n\sum\limits_{p\in\Z} \Theta^{s,k,n-m}_{i+1/2-p} U^{s,m}_{p},\ 
d_{i+1/2}^{s,k,n}:=\Delta x \D t \sum\limits_{m=0}^n\sum\limits_{p\in\Z} \Upsilon^{s,k,n-m}_{i+1/2-p} U^{s,m}_{p},
\end{align*} 
which approximate
$\displaystyle\int_0^{t^n}\int_{\R}\Theta^{s,k}(x_{i+1/2}-\xi,t^n-\tau) U^{s,\D}(\tau,\xi )\d \xi \d \tau$ and 
$\displaystyle\int_0^{t^n}\int_{\R}\Upsilon^{s,k}(x_{i+1/2}-\xi,t^n-\tau) U^{s,\D}(\tau,\xi )\d \xi \d \tau$
respectively. 
 Further, $\Theta_p^{j,k,s}=\mu^{j,k,p} \Gamma^{j,k,s}$ and  $\Upsilon_p^{j,k,s}=\eta^{j,k,p} \theta^{j,k,s},$ where $\mu^{j,k,p},\eta^{j,k,p}$ are the  integral averages of $\mu^{j,k},\eta^{j,k}$ over $C_p$,  and $\Gamma^{j,k,s},\theta^{j,k,s}$ are the integral averages of $\Gamma^{j,k},\theta^{j,k}$ over $C_s$, respectively.
In general, $\mathcal{F}^k$ can be defined as an appropriate nonlocal extension of any monotone numerical flux, meant for local conservation laws, for example,
\begin{equation*}
  \mathcal{F}^k(a_1,a_2,a_3)
   = 
  \frac{a_1}{2}\Big( f^k(a_2)
    +
    f^k(a_3)\Big)
  -
  \beta\frac{(a_3-a_2)}{2\, \lambda}, \beta\in (0,2/3),
\end{equation*} is an extension of Lax-Friedrich's flux. This flux will be used in the sequel,
where  $\Delta t$ is chosen in order to satisfy
the CFL condition
\begin{equation}\label{CFL_LF}
   \lambda \le \frac{\min(1, 4-6\beta,6\beta,1-\Delta t\abs{\boldsymbol{R}}_{(\lip(\R^{2N}))^N})}{1+6\abs{\boldsymbol{f}}_{(\lip(\R))^N}\norma{\boldsymbol{\nu}}_{(L^\infty(\R^N))^N}}.\end{equation}
 We omit the proof of convergence, which can be established via a source-splitting argument combining the results of \cite{ACG2015,AHV2024,AV2026}. However, we state the precise result below for the sake of completeness. {\begin{theorem}[Convergence]\label{convergence}
Assume that (\textbf{H1})--(\textbf{H4}) hold. As $\Delta x \rightarrow 0$, the approximations $\boldsymbol{U}^{\D}$ 
generated by the  marching formula \eqref{scheme2} converge in $(L^1_{\loc}(\overline{Q}_T))^N$ and pointwise a.e.~in $\overline{Q}_T$ to the entropy solution 
	$\boldsymbol{U}^{\D}\in (L^{\infty}({\overline{Q}_T}))^N \cap (C([0,T];L^1(\R{;[0,1]})))^N$ of the Cauchy problem \eqref{nlm}-\eqref{init}
  with initial data $\boldsymbol{U}_0 \in ((L^1 \cap \operatorname{BV})(\R{;[0,1]}))^N$.
\end{theorem}
The above theorems imply that the entropy solution satisfies the following regularity estimates.
\begin{theorem}[Regularity of entropy solutions]\label{tC}
For $0 \le t \le T$, the IVP~\eqref{nlm}--\eqref{init} admits a unique entropy solution $\boldsymbol{U}$ satisfying the following:
\begin{align*}
U^k(t,\cdot)& \geq 0,
    \quad k \in \mathcal{N}, \\
    \|\boldsymbol{U}(t,\cdot)\|_{(L^\infty(\mathbb{R}))^N} 
   & \leq e^{\mathcal{C}\,t\left(1+\|\boldsymbol{U}_0\|_{(L^1(\mathbb{R}))^N}\right)} 
    \|U_0\|_{(L^\infty(\mathbb{R}))^N}, \\
\|\boldsymbol{U}(t,\cdot)\|_{(L^1(\R))^N} &= \|\boldsymbol{U}_0\|_{(L^1(\R))^N},\\
\sum_{k=1}^N \TV(U^k(t,\cdot)) 
&\le e^{\mathcal{C}_{TV} t} \sum_{k=1}^N \TV(U_0^k) + \big(e^{\mathcal{C}_{TV} t} - 1\big),\\
\|\boldsymbol{U}(t_2,\cdot) - \boldsymbol{U}(t_1,\cdot)\|_{(L^1(\R))^N} 
&\le N \mathcal{C}_L |t_2 - t_1|, \quad 0 \le t_1, t_2 \le T,
\end{align*}
$\mathcal{C},C_{TV}$, $C_L$ depends on $|f|_{(\mathrm{Lip}(\mathbb{R}))^N}$, 
$\|\nu\|_{(L^\infty(\mathbb{R}^N))^N}$, 
$\|\mu'\|_{(L^\infty(\mathbb{R}))^{N^2}}$, 
$\|\Gamma\|_{(L^\infty(\mathbb{R}^+))^{N^2}}$, 
and $|R|_{(\mathrm{Lip}(\mathbb{R}^{2N}))^N}$.
\end{theorem}

\section{Numerical Experiments}\label{num1}
We now present 
\begin{figure}[h!]
\centering
\includegraphics[width=0.8\textwidth,keepaspectratio]{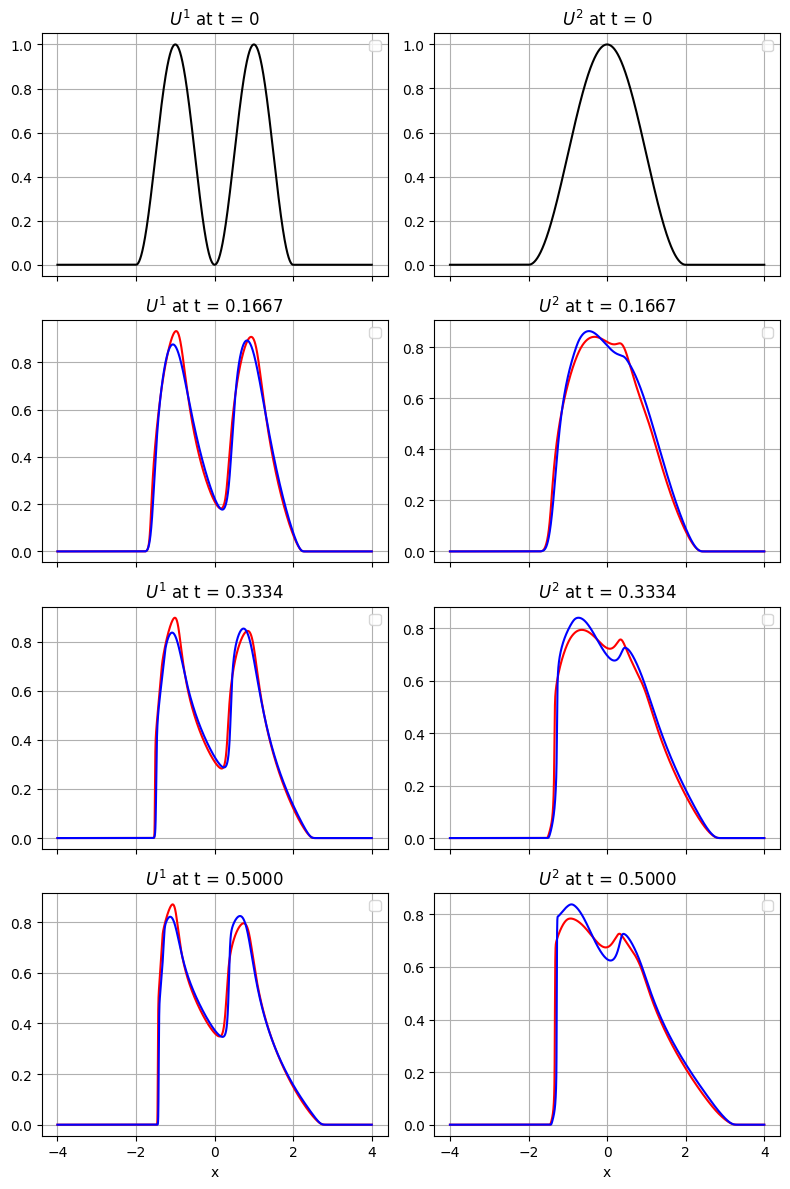}
\caption{Comparison of solutions to \eqref{nlm},\eqref{eq:ex1} ({\color{blue}{---}}) and its nonlocal-space-only counterpart({\color{red}{---}}) at times $t = 0.01667,\,0.3334,\,0.5$.}
\label{fig:ex211}
\end{figure}
\begin{figure}[h!]
\centering
\includegraphics[width=0.8\textwidth,keepaspectratio]{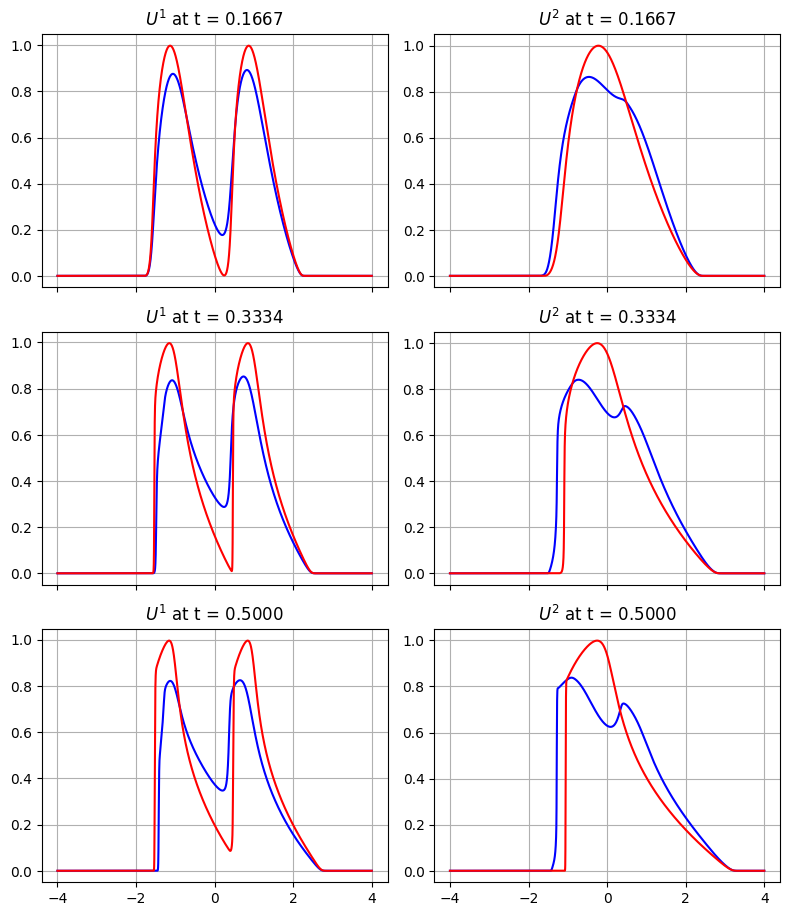}
\caption{Comparison of solutions to \eqref{nlm},\eqref{eq:ex1}({\color{blue}{---}})and its no source counterpart({\color{red}---}) at times $t = 0.01667,\,0.3334,\,0.5$.}
\label{fig:ex2111}
\end{figure}numerical experiments to illustrate the theoretical results of the previous section. In particular, we study qualitative features of solutions, and stability with respect to perturbations of model parameters. 
We compute numerical approximations of \eqref{nlm} on the domain $[-4,4]$ up to the final time $T=0.5$ for $N=2$, using the scheme \eqref{scheme2}. The parameters $\beta = 0.333433$, and $\lambda = 0.1286$ are selected to satisfy the CFL condition \eqref{CFL_LF}. Further, 
$f^1(x)=f^2(x)=xg(x),\nu^1(x)=1.5g(x)$and $
\nu^2(x)=2.5g(x)$ where $g(x)=1-x$.
The spatial and temporal kernels are given by
\[
\mu^{j,k}(x)=\mu(x)=L_1(\delta_x-x)^3\mathbbm{1}_{(0,\delta_x)}(x), \qquad 
\Gamma^{j,k}(t)=\Gamma(t)=L_2(\delta_t-t)^2\mathbbm{1}_{(0,\delta_t)}(t),\ 1\le j,k\le 2,\]
with normalization constants $L_1$ and $L_2$ satisfying
$
\displaystyle\int_{\mathbb{R}} \mu(x)\,dx = 1=
\displaystyle\int_{\mathbb{R}^+} \Gamma(t)\,dt.$  The initial data are
\begin{align}\label{eq:ex1}
U^1_0(x)=\sin^2(0.5\pi x)\mathbbm{1}_{(-2,2)}(x), \qquad
U^2_0(x)=\cos^2(0.25\pi x)\mathbbm{1}_{(-2,2)}(x).
\end{align}
The source term is defined by
\begin{align}\label{eq:Rk}
R^k(\boldsymbol{U},\boldsymbol{U} \circledast \boldsymbol{\Theta})
= S^{k-1}(U^{k-1},U^k,U^{k-1}*\Theta,U^k*\Theta)
- S^k(U^k,U^{k+1},U^k*\Theta,U^{k+1}*\Theta),
\end{align}
where $S^0=S^2=0$ and
\begin{align}\label{eq:Sk}
S^1(a,b,A,B)
= \big(g^2(b)\nu^2(B)-g^1(a)\nu^1(A)\big)^+ a
- \big(g^2(b)\nu^2(B)-g^1(a)\nu^1(A)\big)^- b.
\end{align}This model represents a two-lane unidirectional traffic flow in which the second lane corresponds to faster vehicles, adapted from \cite{AHV2023,HR2019}. The spatial nonlocality models driver anticipation of downstream density, while the temporal convolution in the flux accounts for memory effects in driver response, such as delayed braking or acceleration.

The discretization parameters are $\Delta x = 7.8125 \times 10^{-3},\ \delta_x = 0.078125$, and $\delta_x = 0.4$. Figure \ref{fig:ex211} compares the space–time nonlocal model \eqref{nlm} (blue) with its memoryless counterpart (red). At early times, the memory solution is smoother, reflecting temporal averaging. By $t = 0.5$, it develops sharper fronts, while the memoryless solution diffuses. This reversal shows that accumulated memory feeds past gradients back into the flux, leading to nontrivial dynamics beyond simple smoothing. This behavior is consistent with the nonlinear dependence of the flux on the temporally averaged state.
\begin{figure}[h!]
\centering
\includegraphics[width=0.8\textwidth,keepaspectratio]{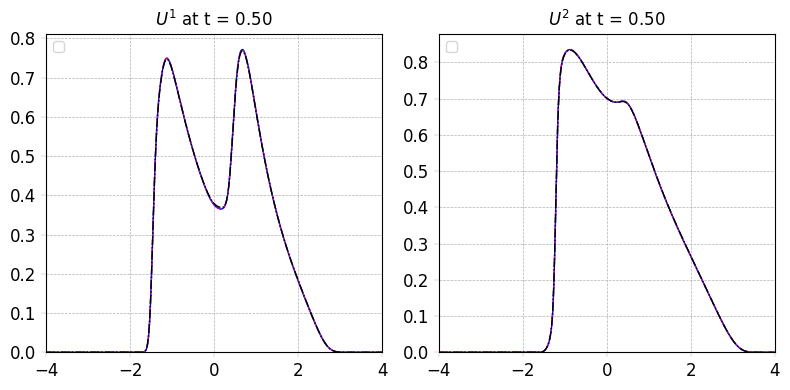}
\caption{Solutions to \eqref{nlm}, \eqref{eq:ex1} corresponding to decreasing perturbation parameter $\epsilon$, illustrating convergence toward the reference solution with perturbation in $\mu$.  $\epsilon=$ 
$1$({\color{black}\chainn})
 $1/2$({\color{red}\chainn}),
$1/4$({\color{magenta}\chainn}), 
$1/8$({\color{magenta}\dashed}), $1/32$({\color{green}\dotted})}.
\label{fig:ex32}
\end{figure}
Figure \ref{fig:ex2111} illustrates the effect of the source term $R^k$ by comparing the full model \eqref{nlm} against the source-free system obtained by setting $R^k=0$. The source term, defined in \eqref{eq:Rk}-\eqref{eq:Sk}, models inter-lane exchange driven by the relative flux differential between the two lanes. Its inclusion produces markedly different dynamics: the two solution components separate more rapidly, the shock in $U^1$ forms earlier and at a different location, and the amplitude of $U^2$ is visibly reduced by $T = 0.5$. This confirms that the source term plays a structurally significant role in the coupled system and cannot be treated as a lower-order perturbation.

We now examine the sensitivity 
 of the solution with respect to perturbations in the model parameters and illustrate the theory presented in the previous section.  
To isolate the effect of each parameter, we perturb one coefficient at a time while keeping all other parameters fixed. In the following experiments, let $\Delta x=0.00625, \eta_x=0.5$ and $\eta_t= 0.8.$ We first consider perturbations of the spatial interaction kernel by defining
\[
\mu_\epsilon(x) = C_\epsilon\, \mu(x)\left(1 + \epsilon \frac{x}{\delta_x}\right)\]
for $\epsilon = 2^{-s}$, $s=1,\dots,6$, where $C_\epsilon$ is chosen such that 
$\displaystyle\int_{\mathbb{R}} \mu_\epsilon(x)\,dx = 1$.
Let $\boldsymbol{U}_\epsilon(T,\cdot)$ denote the corresponding numerical solution at time $T=0.5$ and let $\boldsymbol{U}(T,\cdot)$ denote the reference solution associated with the unperturbed kernel $\mu$.  \begin{figure}[h!]
  \centering \noindent\begin{minipage}{0.46\textwidth}
    \centering
    \begin{tabular}{|c|c|c|c|c|c|c|c|c|c|}\hline
     \multicolumn{1}{|c|}{ $ \epsilon$}&\multicolumn{1}{|c|}{$\displaystyle\frac{e_{\epsilon}(T)}{1000}$}\vline & \multicolumn{1}{|c|}{$\alpha$}\vline\\
     \hline
     $1/2$&$2.08$&\tabularnewline
     \hline
     $1/4$&$1.09$&$0.94$\tabularnewline
     \hline
     $1/8$&$0.56$&$0.97$\tabularnewline
     \hline
     $1/16$&$0.28$&$0.98$\tabularnewline
     \hline
     $1/32$&$0.14$&$0.99$\tabularnewline
     \hline
     $1/64$&$0.0071$&$1.00$\tabularnewline
     \hline
 \end{tabular}
  \end{minipage}
\noindent\begin{minipage}{0.4\textwidth}
\includegraphics[width=.8\textwidth, trim = 40 25 20 5]{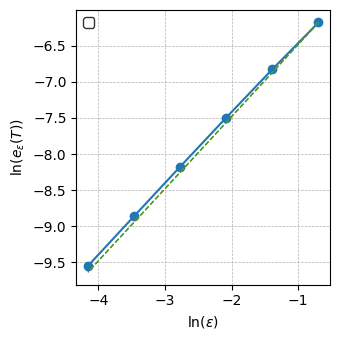}
  \end{minipage}  \caption{Convergence rate $\alpha$ for the decreasing perturbation dynamics with perturbation in $\mu$. Observed convergence rate({\color{blue}\oline}), theoretical convergence rate  ({\color{darkgreen}\dashed}).}\label{fig:ex21}
\end{figure}We measure the deviation using the $L^1$-error
$
e_\epsilon(T) = \|\boldsymbol{U}_\epsilon(T,\cdot) - \boldsymbol{U}(T,\cdot)\|_{(L^1(\mathbb{R}))^N},$
and estimate the convergence rate
$
\alpha = \log_2\!\left(\frac{e_\epsilon(T)}{e_{\epsilon/2}(T)}\right).$ The results, reported in Figure~\ref{fig:ex32}-\ref{fig:ex21}, show that the error decreases approximately linearly with respect to $\epsilon$, with $\alpha \approx 1$. This indicates Lipschitz continuous dependence of the solution on the spatial kernel.

Next, we perturb the velocity functions by defining
\[
g_\epsilon(x) = g(x)\big(1+\epsilon \sin(\pi x)\big), \qquad \epsilon = 2^{-s}, \; s=1,\dots,6,
\]
while keeping $\mu$ and all other parameters unchanged. Denoting the corresponding solution by $\boldsymbol{U}_\epsilon(T,\cdot)$, we again compute the error $e_\epsilon(T)$ and the rate $\alpha$ as above. The numerical results, shown in Figure~\ref{fig:ex22}-\ref{fig:ex31}, exhibit similar behavior, with $e_\epsilon(T)$ decreasing proportionally to $\epsilon$ and $\alpha \approx 1$.

\begin{figure}[ht!]
 \centering
\includegraphics[width=.8\textwidth, trim = 40 25 20 5]{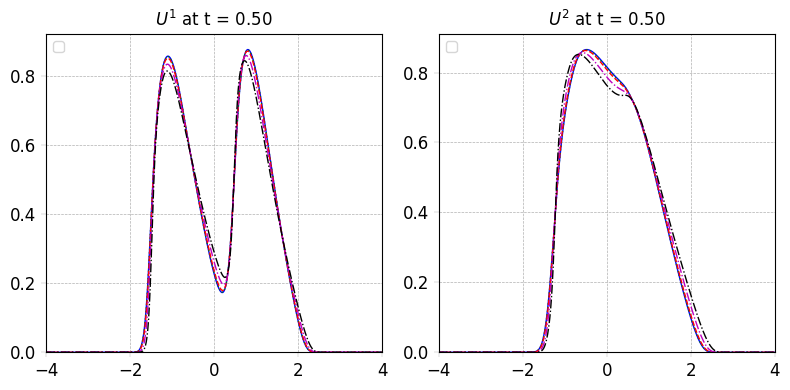}
\caption{Solutions to \eqref{nlm}, \eqref{eq:ex1} corresponding to decreasing perturbation parameter $\epsilon$, illustrating convergence toward the reference solution with perturbation in $g$.  $\epsilon=$ 
$1$({\color{black}\chainn})
 $1/2$({\color{red}\chainn}),
$1/4$({\color{magenta}\chainn}), 
$1/8$({\color{magenta}\dashed}), $1/32$({\color{green}\dotted}).}
  \label{fig:ex22}
\end{figure} \begin{figure}[h!]
  \centering \noindent\begin{minipage}{0.46\textwidth}
    \centering
    \begin{tabular}{|c|c|c|c|c|c|c|c|c|c|}\hline
     \multicolumn{1}{|c|}{ $\epsilon$ }&\multicolumn{1}{|c|}{$\displaystyle\frac{e_{\epsilon}(T)}{10^5}$}\vline & \multicolumn{1}{|c|}{$\alpha$}\vline\\
     \hline
      $1/2$&$7.68$&\tabularnewline
     \hline
     $1/4$&$4.14$&$0.893$\tabularnewline
     \hline
     $1/8$&$2.15$&$0.943$\tabularnewline
     \hline
     $1/16$&$1.10$&$0.971$\tabularnewline
     \hline
     $1/32$&$0.55$&$0.985$\tabularnewline
     \hline
     $1/64$&$0.32$&$0.992$\tabularnewline
     \hline
 \end{tabular}
  \end{minipage}
\noindent\begin{minipage}{0.4\textwidth}
\includegraphics[width=.8\textwidth, trim = 40 25 20 5]{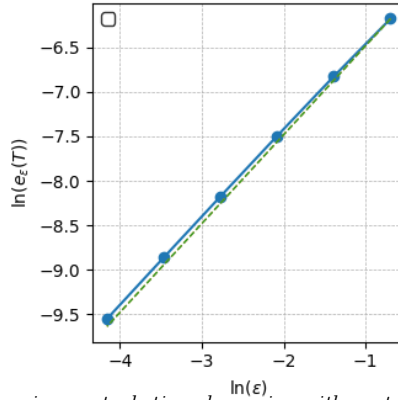}
  \end{minipage}  \caption{Convergence rate $\alpha$ for the decreasing perturbation dynamics with perturbation in $g$. Observed convergence rate({\color{blue}\oline}), theoretical convergence rate  ({\color{darkgreen}\dashed}).}\label{fig:ex31}
\end{figure}

These results confirm the Lipschitz continuous dependence of solutions on model parameters, in agreement with the theoretical stability estimates.  In particular, small perturbations in the coefficients $\mu$ and $g$ lead to proportionally small changes in the solution, confirming the robustness of the model and the numerical scheme. The results with perturbation of the remaining parameters have been observed to have similar results, and are not displayed here.

\bibliographystyle{siam}    
\bibliography{references}
\end{document}